\documentclass[onefignum,onetabnum]{siamart190516}

\usepackage[utf8]{inputenc}
\usepackage[english]{babel}

\usepackage{amsmath, amssymb,mathtools}
\usepackage{dsfont}
\usepackage{cancel}
\usepackage{xfrac}		
\usepackage{mathdots}	
\usepackage{thmtools} 	
\usepackage{tcolorbox}	

\makeatletter
\renewcommand*\env@matrix[1][\arraystretch]{%
  \edef\arraystretch{#1}%
  \hskip -\arraycolsep
  \let\@ifnextchar\new@ifnextchar
  \array{*\c@MaxMatrixCols c}}
\makeatother

\usepackage{algorithmic}

\usepackage{graphicx}
\usepackage{subcaption}
\usepackage{pgfplots}	

\usepackage{color}

\usepackage{tikz}
\usetikzlibrary{calc}
\usetikzlibrary{decorations.markings}
\usetikzlibrary{shapes.geometric,positioning}
\usetikzlibrary{arrows}
\tikzset{vertex/.style={minimum size=2mm,circle,fill=black,draw,inner sep=0pt},
         decoration={markings,mark=at position .5 with {\arrow[black,thick]{stealth}}}}
\pgfplotsset{compat=1.10}
\usepgfplotslibrary{fillbetween}
\usetikzlibrary{backgrounds}
\usetikzlibrary{patterns}

\newcommand{\R}{ \mathbb{R} }
\newcommand{\N}{ \mathbb{N} }

\renewcommand{\l}{ {\lambda_{\CCCom{r}} } }

\newtheorem{assumption}{Assumption}

\newcommand{\dpart}[2]{\frac{\partial #1}{\partial #2}}

\newcommand{\monlambda}{\lambda_{\CCCom{r}}}

\renewcommand{\t}{\top}

\newtheorem{remark}[theorem]{Remark}
\newcommand{\Lim}[1]{\raisebox{0.5ex}{\scalebox{0.8}{$\displaystyle \lim_{#1}\;$}}}

\definecolor{myred}{rgb}{0.86,0.14,0.14}
\definecolor{mygreen}{rgb}{.3,.9,.4}
\definecolor{darkmagenta}{rgb}{0.55, 0.0, 0.55}

\newcommand{\ComM}[1]{{\color{black} #1}}	
\newcommand{\Com}[1]{{\color{black} #1}}		
\newcommand{\CCom}[1]{{\color{black} #1}}	
\newcommand{\CCCom}[1]{{\color{black} #1}}	

\newcommand{\rComM}[1]{{\color{black} #1}}	
\newcommand{\rCom}[1]{{\color{black} #1}}		
\newcommand{\rCCCom}[1]{{\color{black} #1}}	

\title{{Analysis of the Blade Element Momentum theory}}
\author{Jeremy Ledoux\thanks{CEREMADE, CNRS, UMR 7534, Universit\'e Paris-Dauphine, PSL University, France
  (\email{ledoux@ceremade.dauphine.fr},\email{sebastian.reyes-riffo@dauphine.eu}).}
\and Sebasti\'an Riffo
\footnotemark[1]
\and Julien Salomon\thanks{INRIA  Paris,  ANGE  Project-Team,  75589  Paris  Cedex  12,  France
   and Sorbonne Universit\'e, CNRS, Laboratoire Jacques-Louis Lions, 75005 Paris, France (\email{julien.salomon@inria.fr})}.
}

\begin{document}

\maketitle

\begin{abstract}
\CCCom{The Blade Element Momentum theory (BEM) introduced by C.N.H. Lock \emph{et al.} and formulated in its modern form by H. Glauert provides a framework to model the aerodynamic interaction between a turbine and a fluid flow. This theory is either used to estimate turbine efficiency or as a design aid. However, a lack of mathematical interpretation limits the understanding of some of its \rCCCom{issues}. The aim of this paper is to propose an analysis of BEM equations.} Our approach is based on a reformulation of Glauert's model which enables us to identify criteria to guarantee the existence of solution(s), \rCom{ 
 analyze the convergence of usual and new (and more efficient) solution algorithms 
and study turbine design procedures}. The mathematical analysis is 
 completed
 by numerical experiments.
\end{abstract}



\section{Introduction}
\CCCom{Initially introduced to study propellers, } the \textit{Blade Element Momentum} (BEM) \textit{theory} is a model used to evaluate the performance of a propelling or extracting turbine on the basis of its mechanical and geometric parameters as well as the characteristics of the interacting flow. This model results from the combination of two theories: the \textit{Blade Element Theory} and the \textit{Momentum Theory}. The former was introduced by William Froude~\cite{Froude} in 1878 to study the 
 turbines from a local point of view. In this framework, the turbine blade is cut into sections, the \textit{blade elements}, each of them being approximated by a planar model. This approach results in expressions of the forces exerted on the blade element, as functions of the flow characteristics and blade geometry. The fundamental quantities of this model are two experimental coefficients  (usually denoted by $C_L$ and $C_D$), called \textit{lift} and \textit{drag  coefficients}, which account for the forces in the cross-section as functions of the \textit{angle of attack}, i.e. the relative angle between the rotating blade and flow. The results are then integrated along the blade to obtain global values. 

The Momentum Theory, also known as \textit{Disk Actuator Theory} or \textit{Axial Momentum Theory}, was introduced by William J. M. Rankine in 1865~\cite{Rankine} and is, unlike the Blade Element Theory, a theory that adopts a macroscopic point of view  to model the behavior of a column of fluid passing through a turbine.  \CCCom{This approach was later taken up independently first by Nikolay Joukowsky~\cite{Joukowski2,Joukowski3} \rCCCom{(see~\cite{Jouko1})}, Frederick W. Lanchester~\cite{Lanchester} and Albert Betz~\cite{Betz}  to formulate \textit{Betz-Joukowsky Limit}}, which gives the theoretical optimal efficiency of a thin rotor, \CCCom{see~\cite{vanKuikpaper} and~\cite{Okulov2012} for  a description of the three derivations obtained by these authors}. 
A combination of these two approaches was carried out \CCCom{in 1925 by C.N.H. Lock \emph{et. al.}~\cite{Lock1925} and formalized in its modern form in 1926 by Hermann Glauert~\cite{Glauert}}, who also refined the Momentum Theory by including the rotation of the fluid induced by its interaction with the turbine. 

The resulting Blade Element Momentum theory is thus based on two decompositions: (i) a radial	 decomposition of the blades and the fluid column, considered as concentric rings that do not interact with each other, and (ii) a decomposition of the fluid/turbine system into a  macroscopic part via Momentum Theory and a local planar part via Blade Element Theory. Such a description of Glauert's theory is given in the monographs~\cite{Branlard,Burton,Hansen,Manwell,Schaffarczyk,Sorensen,Spera}, see for example
~\cite[p.56]{Hansen} :

{\it``The Blade Element Momentum method couples the momentum theory with
the local events taking place at the actual blades. The stream tube introduced
in the 1-D momentum theory is discretized into $N$ annular elements''}

Though old, Glauert's model is still currently used to evaluate turbine efficiency, as indicated in~\cite[p.23]{Sorensen}:

\textit{``Although a variety of correction models have been
developed since then [...], 
 the momentum theory by
Glauert still remains one \rCCCom{[of the]} the most popular.'' }\\
This longevity can be partly explained by the relative simplicity of the approach compared to the complex phenomenon that develops in the coupled turbine/fluid system. 
This time dependent 3D-fluid/structure interaction problem is a major modeling challenge, which BEM reduces to \Com{0D computations with the help of 2D static data, namely the above mentioned lift and drag coefficients.
 Indeed, these coefficients are obtained by} solving (2D-)partial differential equations, typically stationary Navier-Stokes, or, more often than not, of using experimental data from wind tunnel profile tests. Let us add that the numerical efficiency of this method is all the more crucial as turbine models are mostly implemented as part of design procedures, through iterative optimization loops. In practice, this means that equations have to be solved many times and only simple formulations such as BEM  allows the designer to carry out the computations to a satisfactory point, or at least to provide a good enough initial guess for a finer design, as pointed out in~\cite[p.2]{Ning}:

{\it ``Blade element momentum theory continues to be widely used for wind turbine applications such
as initial aerodynamic analysis, conceptual design, loads and stability analysis, and controls design.''}

\CCCom{Note that other models have been proposed, based on the pioneering work of Joukowsky~\cite{Joukowski}, where the axial wake velocity is assumed to be constant. 
 We refer to~\cite{Jouko2,VanKuik,Jouko1} for an extensive presentation of this model with an historical perspective.}

The significant increase in computational power as well as the theoretical advances obtained in the field of fluid-structure interaction simulation now also make it possible to simulate 3D models based on the Navier-Stokes equation~\cite{Bazilevs,Bazilevs2,Hsu}. \Com{ Alternatively,   BEM model has been combined with other approaches, as Lagrangian
stochastic solvers~\cite{Bossy}, multiple vortex cylinder models~\cite{Branlard_2015} or adapted scaling leading to grid-based variant BEM for large rotors~\cite{Madsen}. 
In view of numerical solving, BEM equations have been reduced to one scalar equation in~\cite{Ning_2014}. A summary of this approach is given in Section~\ref{Sec:Sim_mod}
.} 
Note finally that similar theories have been developed to model vertical axis turbines (of Darrieus type)~\cite{Ingram}.

The purpose of this article is to analyze the Blade element momentum theory 
 from a mathematical point of view. 
The results we obtain in the course of this analysis elucidate issues related to the well-posedness of the model, the numerical solution of its equations and the optimality of a blade design. 
 They concern two versions of Glauert's model, which we call \textit{Simplified model} and \textit{Corrected model}. The former allows us to illustrate the main features of our approach, whereas the latter  includes some corrections usually considered to remedy the mismatch between the simplified model and experimental observations.

The paper is structured as follows: Section~\ref{sec:bem} provides a brief exposition of the derivation of the model. We then focus on the resulting algebraic system. 
The key point of our analysis is to look deeper into Glauert's macroscopic-local decomposition to reformulate these equations into a single equation containing two very distinct terms: a universal term, independent of the turbine under consideration and associated with the macroscopic part of the model, and an experimental term, which depends on the characteristics of the blades and is associated with the local part of the model. In this context, we show that solving the equations associated with Glauert's model actually means finding an angular value that makes these two terms equal. 
This result agrees with some implicit conclusions reported before   
  that were not formalized mathematically and rather used for pedagogic purpose, such as 
in~\cite[Figure 3.27, p.126]{Manwell}. 
In contrast, our analysis 
 gives rise to new theoretical and numerical results. In Section~\ref{Sec:existence}, our reformulation enables us to identify explicitly which   
 assumptions related to the turbine parameters can guarantee the existence of a solution.
 In addition, we obtain 
a classification of multiple solution cases based on the modeling assumptions. 
In Section~\ref{Sec:algo}, we present the usual solution algorithm\ComM{s} and derive from our approach more efficient procedures.
As already mentioned, BEM is also used in turbine design, where the simplified model is included in a specific optimization method.
We recall the details of the resulting procedure in Section~\ref{Sec:opti} and describe an optimization algorithm for the corrected model.
We finally present some numerical experiments in Section~\ref{numres}. 

Our results are based on assumptions related to physical parameters, e.g., on the coefficients $C_D$ and $C_L$. 
We do not claim that these assumptions are necessary. However, we treat $C_D$ and $C_L$ as generic functions (endowed with general properties), hence it is often possible to find examples which make our assumptions optimal.

In what follows, we denote by $\R^+$ and $\R^-$  the sets  of  positive real numbers  and non-positive real numbers, respectively.

\section{The blade element momentum theory}\label{sec:bem}

In this section, we present the model proposed by Glauert to describe the interaction between a turbine and a flow. After having introduced the relevant variables, we recall the main steps of the reasoning leading to the equations of the model. We then detail the two versions of the model considered in this paper.

\subsection{Variables}

The \rCCCom{blade element momentum theory} 
aims to establish algebraic relations that characterize the interaction between a flow and a rotating blade, named \textit{turbine} in what follows. In this way, Glauert's model couples two descriptions: a global macroscopic model that describes the evolution of fluids rings crossing the turbine,  and a local one, that summarizes in 2D the behavior of a section of a blade, \textit{a blade element}, under the action of the fluid.

\Com{The flow is supposed to be constant in time and incompressible.} The latter assumption implies that the flow velocities in the left and right neighborhoods of the turbine have a same value $U_{0}$. We denote by $U_{-\infty}$ and $U_{+\infty}$ the upstream and downstream velocities, respectively. \Com{Though not considered in this paper, tangential velocity can also be studied. In particular, a jump of this variable caused by the actuator disk is often reported and can be modeled in the framework of the momentum theory~\cite[Chap. 9]{Branlard}.}
As the BEM model does not take into account interactions between blade elements \Com{and assumes that $\Omega$ and $U_\infty$ are constant},  we consider in this paper 
a fixed blade element and 
a fixed value of \CCCom{the \emph{local speed ratio} 
$\lambda_{\CCCom{r}}:= \frac{\Omega r}{U_{-\infty}}.$ 
 Here, $r$ is the distance of the element to the rotation axis, with $r\leq R$, where $R$ is the radius of the blade.  
In practical cases, the turbine works \emph{at constant Tip Speed Ratio (TSR)}: 
 $\Omega$ is indeed often controlled through the torque exerted by a generator in such 
a way that the ratio TSR$:=\frac{\Omega R}{U_{-\infty}}$ is kept constant for various values of ${U_{-\infty}}$}. It follows that the value of $\l$ associated with one element only depends on $r$. In the sequel, we consequently use the variable $\l$ to describe the location of a blade element.

\subsubsection{Macroscopic variables and BEM unknowns}

Glauert's model ultimately consists of a system which links together three variables $a,a'$ and $\varphi$ associated with a ring of fluid. 
The two former are usually called the {\it axial} and {\it angular induction factor}, respectively. They are defined by
\begin{equation}\label{defa}
a := \frac{U_{-\infty} - U_{0}}{U_{-\infty}},\quad a':=\frac{\omega}{2\Omega},
\end{equation}
where $\omega$ is the rotation speed of the considered ring of fluid. 
 The angle $\varphi$ is  the \textit{relative angle deviation} (see~\cite[p.120]{Manwell}) of the ring, meaning that:

\begin{equation}
\tan\varphi=\frac{1-a}{\lambda_{\CCCom{r}}(1+a')}.\label{Eq1}
\end{equation}
For the sake of simplicity, and to emphasize their role of unknowns in Glauert's model, we omit the dependence of $a,a',\varphi$ (and $\alpha$ in what follows) 
on $\lambda_{\CCCom{r}}$ in the notation. 

\subsubsection{Local variables}\label{locvar}
Let us denote by $U_{rel}$ the relative fluid speed (also called \textit{apparent fluid speed}) perceived from this blade element while rotating. By definition of $\varphi$, we have:
\begin{equation}\label{eq:urel}
U_{rel}=\frac{U_0}{\sin\varphi}.
\end{equation}
\Com{This variable is not defined when $\varphi=0$ and, as an intermediate quantity, will not appear in the final model. However, the limit case $\varphi\rightarrow 0$ is discussed in Section~\ref{Sec:correc_model}.}
For a given blade profile, the {\it lift} and {\it drag coefficients} $C_L$ and $C_D$  are defined by 
\begin{equation}\label{dLdT}
dL= C_L(\alpha)\frac 12 \rho U_{rel}^2 c_\l dr, \qquad dD= C_D(\alpha)\frac 12 \rho U_{rel}^2 c_\l dr,
\end{equation}
where $\rho$ is the mass density of the fluid, $dL$ and $dD$ are the elementary lift and drag forces applying to a blade element of thickness $dr$ and of chord \Com{$c_\l$}.
The parameter $\alpha$ is called {\it angle of attack} and defined as the angle between the chord and flow direction, hence satisfies the relation 
\begin{equation}\label{eq:def_alpha}
\alpha=\varphi - \gamma_\l, 
\end{equation}
where $-\pi/2 < \gamma_\l < \pi/2$ 
is the {\it twist} (also called \textit{local pitch}) angle of the blade. The parameters associated with a blade element are summarized in Figure~\ref{fig:angles}. 


\begin{figure}
\centering
\input{figure_angle3}
\caption{Blade element profile and associated  angles, velocities and forces.}\label{fig:angles}
\end{figure}

The coefficients $C_L$ and $C_D$ correspond to the ratio between the lift and drag forces and the dynamic force, i.e., the force associated with the observed kinetic energy. They are determined by the profile of the blade. Once this one is fixed, the main design parameters are $c_\l$ and $\gamma_\l$ whose optimization is discussed in Section~\ref{Sec:opti}.


The coefficients $C_L$ and $C_D$ are assumed to be known as functions of $\alpha$ and occasionally of Reynolds number ($Re$). 
 The latter case is indeed rarely considered in the monographs, where $Re$ is assumed to be constant with respect to $\alpha$ as soon as $U_{-\infty}$, $\l$, $\Omega$, $c_\l$ are fixed. For the sake of simplicity, we also neglect the changes in $Re$ in this paper. However, our results can be extended to non-constant Reynolds numbers~\cite[p.374-375]{Spera}, i.e. in situations where the functions $(\alpha,Re)\mapsto C_L(\alpha,Re)$ and $\alpha\mapsto C_D(\alpha,Re)$ have to be taken into account. \Com{ In practice, $Re$ is \rCCCom{either assumed to be known a priori and used to select the corresponding $C_L$ and $C_D$, or dealt with iteratively together with $C_L$ and $C_D$ to get a more accurate result}. } 
Examples of variations of $C_L$ and $C_D$ with respect to $Re$ are given in~\cite[p.169]{Burton}.

Though changing from one profile to another, the behaviors of $C_L$ and $C_D$ as functions of $\alpha$ can be described qualitatively in a general way. 
The coefficient $C_L$ usually increases nearly linearly with respect to $\alpha$ up to a given critical angle $\alpha_s$, with $0<\alpha_s<\pi/2$, where the so-called \textit{stall} phenomenon occurs: $C_L$ then decreases rapidly (see, e.g.,~\cite[p.93-94]{Branlard} and~\cite[p.375]{Spera}), causing a sudden loss of lift.
For $C_D$ is associated with a drag force, it is always positive and defined for all angles. 
 This coefficient usually slightly increases with $\alpha$ up to $\alpha = \alpha_s$, and then becomes very large. \Com{
Though most designs do not prevent the inner part of the blade to generate stall, the condition $\varphi-\gamma_\l<\alpha_s$ is often considered in the blade design phase.} \CCom{Note finally that the greater part of our analysis applies for angles of attack such that $C_L$ is positive. }
The properties of $C_L$ and $C_D$ required for our analysis are summarized in the following assumption.
\begin{assumption}\label{assume:gal}
\CCom{For some $\beta\in 
\R^+$
, }the function $\alpha\mapsto C_L(\alpha)$ is continuous on 
 $I_\beta:=[-\beta,\beta]$,
and positive on $I_\beta\cap\R^+$. The function $\alpha\mapsto C_D(\alpha)$ is defined, continuous and non-negative on $\R$.
\end{assumption}
\Com{As a consequence, $C_L(0)$ is assumed to be positive, which is equivalent in practice to the fact that the angle of attack corresponding to zero lift is negative. This assumption is true for usual designs.}

\subsection{Glauert's modeling}
For the sake of completeness, we now shortly recall the reasoning proposed by Glauert to model the interaction  between a turbine and a flow. We refer to~\cite[Chap.~3]{Burton}, for a more extended presentation of this theory. 
We denote by $dT$ and $dQ$ the infinitesimal thrust and torque that apply on the blade element of thickness $dr$ under consideration.

\subsubsection{Macroscopic approach}
The first part of the model is related to the Momentum Theory and deals with the macroscopic evolution of a ring of fluid. It aims to express $dT$ and $dQ$ in terms of $a,a'$ and $\varphi$.

Denote by $p_-$ and $p_+$ the fluid pressures on the left and right neighborhoods of the blade, respectively. Applying Bernouilli's relation between $-\infty$ and $0^-$ and between  
$0^+$ and $+\infty$ gives rise to $p_- - p_+ = \frac 12\rho(U_{-\infty}^2-U_{+\infty}^2 )$.
Considering then the rate of change of momentum on both sides of the turbine, we get a second expression for the variation in the pressure, namely $p_- - p_+ = \rho(U_{-\infty}-U_{+\infty} )U_0$.
Combining the two previous equations and using~\eqref{defa}, we obtain $U_{+\infty}=(1-2a)U_{-\infty}$.
Since  \CCom{$dT=(p_- - p_+) 2\pi r   dr$ } and $dQ =  \omega \rho U_0 2\pi r^3dr$, we finally get
\begin{align}
dT&=   C_T(a) U_{-\infty}^2\rho \pi r          dr,\label{defT1}\\
dQ&=   4a'(1-a) \lambda_{\CCCom{r}} U_{-\infty}^2  \rho \pi r^2  dr.\label{defQ1}
\end{align}
where 
$
C_T(a):=\frac{dT}{\frac 12 U_{-\infty}^2\rho 2\pi r dr}=4a(1-a),$
is the {\it local thrust coefficient}~\cite{Wilson}.

\subsubsection{Local approach}\label{sec:locexpr}
Another set of equations can be obtained via the Blade Element Theory, where local expressions for infinitesimal thrust and torque are considered.  
The reasoning consists in combining the elementary lift and drag~\eqref{dLdT} expressed in the rotating referential  with~\eqref{eq:urel}. This gives
\begin{align}
dT&=   \sigma_\l\frac{ (1-a)^2}{\sin^2\varphi} \left(C_L(\alpha)\cos\varphi + C_D(\alpha)\sin\varphi\right) U_{-\infty}^2\rho \pi r   dr,\label{defT2} \\
dQ&=   \sigma_\l\frac{ (1-a)^2}{\sin^2\varphi} \left(C_L(\alpha)\sin\varphi - C_D(\alpha)\cos\varphi\right) U_{-\infty}^2\rho \pi r^2 dr,\label{defQ2}
\end{align}
where $\sigma_\l:=\frac{Bc_\l}{2\pi r}$, with $B$ is the number of blades of the turbine.


\subsubsection{Combination of local and global approaches }
To get a closed system of equations, Glauert combined the results of the two last subsections. More precisely, equating~\eqref{defT1} and~\eqref{defQ1}  with~\eqref{defT2} and~\eqref{defQ2}, respectively, using~\eqref{eq:def_alpha} and dividing both resulting equations by $4(1-a)^2$ gives
\begin{align}
\frac{a }{1-a}&=\frac{\sigma_\l}{4        \sin^2\varphi}\left(C_L(\varphi-\gamma_\l)\cos\varphi + C_D(\varphi-\gamma_\l) \sin\varphi\right) \label{Eq2},\\
\frac{a'}{1-a}&=\frac{\sigma_\l}{4\lambda_{\CCCom{r}} \sin^2\varphi}\left(C_L(\varphi-\gamma_\l)\sin\varphi - C_D(\varphi-\gamma_\l) \cos\varphi\right) \label{Eq3}.
\end{align}
The system obtained by assembling~\eqref{Eq1},~\eqref{Eq2} and~\eqref{Eq3} is the basis of Glauert's Blade Element Momentum theory.

\subsection{Simplified model}

In the monographs devoted to aerodynamics of wind turbines, the contribution of $C_D$ is sometimes set to zero.  This point is discussed in~\cite[p.135]{Sorensen}, where it is particular stated:

\textit{``Since the drag force does not contribute to the induced velocity physically, $C_D$
is usually omitted when calculating induced velocities.''}.\\
In the same way, Manwell and co-authors mention in~\cite[p.125]{Manwell}:

\textit{``In the calculation of induction factors,[...] accepted practice is to set $C_D$ equal to zero
[...]. For airfoils with low drag coefficients, this simplification
introduces negligible errors.''}\\
This assumption is actually justified in many cases, since the procedures used to design profiles minimize their drag. 
As a matter of fact, the usual blade design procedure starts by selecting a twist angle $\gamma_\l$ minimizing the ratio $\frac{C_D}{C_L}$, see Section~\ref{sec:design}.


We also consider the case where $C_D=0$, and referred to as {\it simplified model} in the following. In view of~\eqref{Eq1}, \eqref{Eq2} and~\eqref{Eq3}, it corresponds to the three equations:

\begin{align}
\tan\varphi&=\frac{1-a}{\lambda_{\CCCom{r}}(1+a')}, \label{Standard:Eq1}\\
\frac{a }{1-a}&=\frac{1}{\sin^2\varphi} \mu_L(\varphi)\cos\varphi,  \label{Standard:Eq2}\\
\frac{a'}{1-a}&=\frac{1}{\lambda_{\CCCom{r}} \sin\varphi} \mu_L(\varphi) \label{Standard:Eq3},
\end{align}
where we have introduced the dimensionless function
 $\mu_L(\varphi):=\frac{\sigma_\l}{4 } C_L(\varphi-\gamma_{\monlambda})$,
which is defined on 
 $I_{\beta,\gamma_\l}:=[-\beta+\gamma_\l,\beta+\gamma_\l]$, 
 by virtue of Assumption~\ref{assume:gal}.

\subsection{Corrected model}\label{Sec:correc_model}

To get closer to experimental results, many  modifications of the model (\ref{Eq1},\ref{Eq2},\ref{Eq3}) have been introduced, see e.g.~\cite[Chapter 7]{Sorensen}.  Hereafter, we present three important corrections, namely non-zero drag coefficient $C_D$, tip loss correction and a specific treatment of large values of $a$. The first and the last will modify significantly the analysis developed for the simplified model.


\subsubsection{Slowly increasing drag}\label{Slow_inc_Drag}

In addition to consider $C_D$ strictly positive, we shall assume in some parts of the analysis 
a slow increasing of this parameter from \rComM{$0$ up to} the occurrence of the stall phenomenon. 

\subsubsection{Tip loss correction}






The equations of Momentum Theory are derived assuming that \Com{the turbine can be modeled as an actuator disk}.
Such a framework corresponds to a rotor with an infinite number of blades.
However, in real life situations, a modification of the flow at the tip of a blade has to be included to take into account that the circulation of the fluid around the blade must go down (exponentially) to zero when $r\rightarrow R$, where $R$ is the turbine radius. In this way, 
Glauert (see~\cite[p.268]{Glauert3}) introduced \CCCom{an approximation of  the Prandtl tip function $F_\l$~\cite{Prandtl} (see also~\cite[p.240]{Branlard})}:
\begin{equation*}
F_\l(\varphi):=\frac{2}{\pi} \CCom{\arccos}\left( \exp(-\frac{B/2(1-\frac{\l U_{-\infty}}{\Omega R})}{(\frac{\l U_{-\infty}}{\Omega R})\sin\varphi}) \right)=\frac{2}{\pi} \CCom{\arccos}\left( \exp(-\frac{B/2(1-r/R)}{(r/R)\sin\varphi}) \right), 
\end{equation*}
as a supplementary factor in~\eqref{defT1} and~\eqref{defQ1}. This modification gives rise to
\begin{align}
dT&=   4a (1-a)F_\l(\varphi) U_{-\infty}^2 \rho \pi r          dr,\label{defT1_tip}\\
dQ&=   4a'(1-a)F_\l(\varphi) U_{-\infty}   \rho \pi r^3 \Omega dr.\label{defQ1_tip}
\end{align}
\Com{
\rCCCom{Note that some authors include the tip loss factor only in~\eqref{defT1}, see~\cite{Clifton}}. The results of our paper can readily be adapted to this version of the model. }
Further models of tip loss correction have been introduced in between.   
\CCCom{We refer to~\cite[Chap. 13]{Branlard}, and~\cite{Shen} for reviews of other models. An alternative approach based on Extended vortex theory has also been proposed in~\cite{Wood}.}


\subsubsection{Correction for high values of $a$}\label{sec:corr_a}
 For induction factors $a$ larger than about $0.4$ (see \cite[p.297]{Spera}), a turbulent wake usually appears, and it is broadly considered that momentum theory does not apply. This fact was already reported by Glauert (see~\cite{Glauert26}
), who proposed to modify $C_T(a)$ in~\eqref{defT1} when $a$ becomes larger than a given threshold $a_c$. 
Subsequently, many other expressions have been proposed to fit better with experimental data, see~\cite[Section 10.2.2]{Branlard}.
All these variants lead to a new expression for $dT$ that reads
\begin{equation}
dT=   4\left(a(1-a) + \psi\left( (a-a_c)_+\right)\right)F_\l(\varphi) U_{-\infty}^2 \rho \pi r          dr\label{defT1Gen}.
\end{equation}
where $(a-a_c)_+ :=\max\{0,a-a_c\}$ and $\psi$ is a given function defined on $\R^+$.
Some corrections are presented via the function $\psi$ in Table~\ref{Table_correc}. Glauert's empirical correction is obtained by combining experimental data and the constraints $C_T(a_c)=4a_c(1-a_c)$ and $C_T(1)=2$. This leads to a discontinuity at $a=a_c$ when  $F_\l(\varphi)\neq 1$. 
 Buhl proposed in~\cite{Buhl} a modification to fix this issue.

\begin{table}[h!]
\begin{center}
\begin{tabular}{|c||c|c|}\hline
Authors & $a_c$ & $\psi((a-a_c)_+)$ \\ \hline\hline  
\begin{tabular}{cc}
Glauert \cite{Glauert3},
\cite[p.53]{Hansen} 
\end{tabular} & 
1/3	& 
$\dfrac{(a-a_c)_+}{4}\left( \dfrac{(a-a_c)^2_+}{a_c} + 2(a-a_c)_+ + a_c \right)$	
\\ \hline
\begin{tabular}{@{}l@{}} 
 Empirical Glauert   \\ 
\cite{Eggleston,Buhl} 
\end{tabular} & 
2/5 		&  
$\approx \left(\frac{1}{2(1-a_c)}-\frac\nu 4(1-a)\right)(a-a_c)_+$ with $\nu=5.5563$
\\ \hline
Buhl	\cite{Buhl}	& 
2/5	& 
$\dfrac 1{2F_{\l}(\varphi)}\left(\dfrac{(a-a_c)_+}{1-a_c}\right)^2$ 
\\ \hline  
\begin{tabular}{@{}l@{}}
	Wilson et al., \\ Spera  \cite[p.302]{Spera}
\end{tabular} 
& 
1/3	& 
$(a-a_c)_+^2$
\\ \hline
\end{tabular}
\caption{\rCCCom{Various corrections that can be introduced in the relation between the thrust
coefficient and the induction
factor~\eqref{defT1_tip} in case of turbulent state.
}
}\label{Table_correc}
\end{center}
\end{table}
\subsubsection{On 3D effects}
\CCom{

In order to take into account 3D effects of the rotating blade, correction formulas 
have been proposed~\cite{snel1993, Du1998, Hansen2000} (see the extensive review in~\cite{Sorensen2011}). 
As an example, the following expression can be used to correct $C_L$:
\[C_L^c=C_L+a (\frac {c_\l} r)^b \Delta C_L,\]
where $a\in [2,3]$ and $b\in [1,2]$ are constants and $\Delta C_L=C_{L,inv}-C_L$. In the latter expression, $C_{L,inv}$ is the lift obtained by considering an inviscid flow, i.e., by solving a Laplace problem. A similar expression
can be used for the drag coefficient \rComM{$C_D$}. For simplicity, we do not consider 3D effects in our study, but all the following results can be readily adapted by replacing $C_L$ and $C_D$ by their corrected expressions, as done in the next section with tip-loss correction.
Note finally that the 3D correction presented in~\cite{Du1998} is included in the model considered in Section~\ref{numNL}.
}

\subsubsection{Corrected system}

We repeat the reasoning used to obtain~(\ref{Standard:Eq1}--\ref{Standard:Eq3}), that is, we equalize~\eqref{defT1Gen} and~\eqref{defQ1_tip} with~\eqref{defT2} and~\eqref{defQ2}, respectively. We get:
\begin{align}
\tan\varphi&=\frac{1-a}{\lambda_{\CCCom{r}}(1+a')}, \label{Correc:Eq1}\\
\frac{a }{1-a}&=\frac{ 1}{\sin^2\varphi} (\mu^c_L(\varphi)\cos\varphi + \mu^c_D(\varphi) \sin\varphi)-\frac{\psi\left( (a-a_c)_+\right)}{(1-a)^2},\label{Correc:Eq2}\\
\frac{a'}{1-a}&=\frac{ 1}{\lambda_{\CCCom{r}} \sin^2\varphi} (\mu^c_L(\varphi)\sin\varphi - \mu^c_D(\varphi) \cos\varphi)\label{Correc:Eq3},
\end{align} 
where we have introduced the dimensionless functions 
$\mu^c_L(\varphi):=\frac{\sigma_\l}{4 F_\l(\varphi)} C_L(\varphi-\gamma_{\monlambda})$, and $\mu^c_D(\varphi):=\frac{\sigma_\l}{4 F_\l(\varphi)} C_D(\varphi-\gamma_{\monlambda})$,
defined on $I_{\beta,\gamma_\l}$ 
 and $\R$, respectively. 
 \ComM{The corrected and simplified models coincide} when $F_\l(\varphi)=1$, $a_c=1$ and $C_D=0$. 
\ComM{\begin{remark}
Combining~\eqref{Correc:Eq3} with~\eqref{Correc:Eq1}, one obtains
\begin{equation}
\frac{a'}{1+a'}=\frac{ 1}{ \sin\varphi  \cos\varphi} (\mu^c_L(\varphi)\sin\varphi - \mu^c_D(\varphi) \cos\varphi)\label{Correc:Eq3_alter},
\end{equation}
which is sometimes considered instead of~\eqref{Correc:Eq3} to define $a'$.
\end{remark}
}

\section{Analysis of Glauert's model and existence of solution}\label{Sec:existence}
In this section, we reduce each of the two previous versions of Glauert's model to a single scalar equation. With a view to obtaining existence results,
this leads us to formulate assumptions related to the characteristics of the turbine.
To simplify notation, we introduce the angle $\theta_{\monlambda} \in (0,\frac\pi 2)$ defined by 
\begin{equation}\label{def:theta_lambda}
\tan\theta_{\monlambda}:=\frac 1\lambda_{\CCCom{r}},
\end{equation}
and the intervals
\begin{equation}
I:=I_{\beta,\gamma_\l}\cap(-\frac{\pi}2+\theta_\l,\frac{\pi}2+\theta_\l),\quad
I^+:=I\cap(0,\theta_\l].\label{defI}
\end{equation}
\Com{We comment about these definitions in the next section.}







\subsection{Simplified model}\label{Sec:Sim_mod}
In the setting of the simplified model, a reformulation of~(\ref{Standard:Eq1}--\ref{Standard:Eq3}) can be obtained after a short algebraic manipulation.
\begin{theorem}\label{lem:reform_standard}
Suppose that Assumption~\ref{assume:gal} holds and that $(\varphi,a,a')\in I-\{0,\frac{\pi}2\}\times\R-\{1\}\times\R-\{-1\}$
satisfies~(\ref{Standard:Eq1}--\ref{Standard:Eq3}).
Then $\varphi$ satisfies 
\begin{equation}\label{eq:mu_mu}
\mu_L(\varphi)=\mu_G(\varphi),
\end{equation}
where 
$\mu_G(\varphi):=\sin\varphi\tan(\theta_{\monlambda}-\varphi)$.
Reciprocally, suppose that $\varphi\in I-\{0,\frac{\pi}2\}$ 
satisfies~\eqref{eq:mu_mu} and define $a$ and $a'$ as the corresponding solutions of~\eqref{Standard:Eq2} and \eqref{Standard:Eq3}, respectively. Then
$(\varphi,a,a')\in I-\{0,\frac{\pi}2\}\times\R-\{1\}\times\R-\{-1\}$ 
 satisfies~(\ref{Standard:Eq1}--\ref{Standard:Eq3}).
\end{theorem}
Note that~\eqref{eq:mu_mu} appears -- up to a factor -- in~\cite[p.128, Fig. 3.85a]{Manwell}. 
Some concrete examples of graphs of $\mu_L$ and $\mu_G$ are given in Section~\ref{numintro}.

We have excluded the angles $\varphi=\frac{\pi}2$ and $\varphi=0$ for the sole reason that~(\ref{Standard:Eq1}--\ref{Standard:Eq3}) are not defined for these angle values. However, $\varphi=0$ is naturally associated with the case $a=1$, as it appears in~\eqref{Standard:Eq2}. On the other hand, the value $\varphi=\frac{\pi}2$, that belongs to $I$ if $\beta+\gamma_\l>\frac\pi2 > -\beta+\gamma_\l$, is neither a solution of~(\ref{Standard:Eq1}--\ref{Standard:Eq3}) nor of~\eqref{eq:mu_mu}: setting this value in the former system leads indeed to $a'=-1$, $a=0$ and $-\l=\mu_L(\frac{\pi}2)$ which corresponds to a negative lift, hence contradicts $C_L\geq 0$ on $I_\beta\cap\R^+$ in Assumption~\ref{assume:gal}. 
As a matter of fact, $\mu_G$ is well-defined at these values, which generally do not give rise to solution of~\eqref{eq:mu_mu}.
Finally, note that the right-hand side of~\eqref{eq:mu_mu} is not defined in  the values $\varphi=\pm\frac{\pi}2+\theta_\l$. However, they do not correspond to any solution of~(\ref{Standard:Eq1}--\ref{Standard:Eq3}): inserting them in the last equations leads to $\mp \l= 0.\mu_L(\pm\frac\pi2)$ which contradicts $\l=\frac{\Omega r}{U_{-\infty}}>0$. For all these reasons, the formulation~\eqref{eq:mu_mu} will be considered on the whole interval $I$ in the rest of this paper.
\begin{proof}
Suppose that $(\varphi,a,a')\in I\rComM{-\{0,\frac{\pi}2\}}\times\R-\{1\}\times\R-\{-1\}$ 
 satisfies~(\ref{Correc:Eq1}--\ref{Correc:Eq3}). 
Eliminating $a$ and $a'$ in~\eqref{Standard:Eq1} using~\eqref{Standard:Eq2} and~\eqref{Standard:Eq3}, we get:
\begin{align*}
\tan^{-1}\varphi&=\lambda_{\CCCom{r}}\frac{1+a'}{1-a} 
=\lambda_{\CCCom{r}}\left(1+\frac{\cos\varphi}{      \sin^2\varphi} \mu_L(\varphi) \right) + \frac{1}{\sin\varphi} \mu_L(\varphi).
\end{align*}
so that~\eqref{eq:mu_mu} follows from the definition~\eqref{def:theta_lambda} of $\theta_\l$. Repeating these steps backward ends the proof of the equivalence.
\end{proof}

This result shows that Glauert's model -- here in its simplified version -- essentially boils down to one scalar equation. 
Indeed, suppose that $\varphi$ satisfies~\eqref{eq:mu_mu}, then $a$ and $a'$ can be post-computed thanks to~(\ref{Standard:Eq2}--\ref{Standard:Eq3}). 
\Com{ A similar conclusion has also been obtained by Ning in~\cite{Ning}, but his approach results in another equation that is 
\begin{equation}\label{eqNing}
\frac{\sin\varphi}{1-a} - \frac{\cos\varphi}{\lambda_{\CCCom{r}}(1+a')}=0,
\end{equation}
where $a$ and $a'$ read as functions of $\varphi$ by~\eqref{Correc:Eq2} and~\eqref{Correc:Eq3_alter}, respectively. Note that~\eqref{eq:mu_mu} and~\eqref{eqNing} have the same singularities, namely $0$ and $\pi/2$. However, our approach 
gives rise to a specific physical interpretation and mathematical analysis. More precisely,
}
an important property of~\eqref{eq:mu_mu} is that its left-hand side corresponds to the local description of the problem related to Blade Element Theory, whereas its right-hand side is related to the macroscopic modeling arising from Momentum Theory. 
As a consequence, $\mu_G$ reads as a universal function of fluid-turbine dynamics depending only on $\theta_\l$ \Com{and related to Momentum theory. On the contrary, $\mu_L$ reads as a function which strictly depends on  the turbine under consideration, i.e., on its design parameters $\gamma_\l$ or $\sigma_\l$ as well as its 2D experimental data, through $C_L$, hence, rather relates to Blade element theory}. In this view,~\eqref{eq:mu_mu} is in line with the approach considered by Glauert. In the same way, the two intervals defining $I$, namely $I_{\beta,\gamma_\l}$ and $(-\frac{\pi}2+\theta_\l,\frac{\pi}2+\theta_\l)$ play similar roles in the local and in the macroscopic descriptions as they  correspond to the domains of definition of $\mu_G$ and $\mu_L$, respectively, whereas $I_{\beta,\gamma_\l}\cap \R^+$ and $(0,\theta_\l]$, whose intersection $I^+$ corresponds to 
 positive lift in the two descriptions.

The formulation given in Theorem~\ref{lem:reform_standard} can be used to establish criteria to ensure existence of solution of~\eqref{eq:mu_mu}: existence indeed holds as soon as the graphs of $\mu_G$ and $\mu_L$ intersect. As an illustration, we give a simple condition in the case of symmetric profiles. We express the assumptions in terms of $\mu_L$ to make it coherent with the formulation~\eqref{eq:mu_mu} ; they can however easily be formulated in terms of $C_L$ and $\sigma_\l$.

\begin{corollary} \label{th:exist_standard}

In addition to Assumption~\ref{assume:gal}, suppose that \rComM{$\gamma_{\monlambda}\leq \theta_\l$, that}  the profile under consideration is symmetric with $\gamma_\l>0$, and that 
\begin{equation}\label{assum:mu}
\mu_G(\max I^+)\leq\mu_L(\max I^+),
\end{equation}
where $\max I^+:=\min\{\theta_\l,\beta + \gamma_{\monlambda}\}$, see~\eqref{defI}.
Then~\eqref{eq:mu_mu} admits 
a solution corresponding to a positive lift in $[\gamma_{\monlambda},\max I^+]$.
Moreover, if $\max I^+=\theta_\l$, i.e. $\theta_\l\leq\beta + \gamma_{\monlambda}$, then~\eqref{assum:mu} is necessarily satisfied.
\end{corollary}
\begin{proof}
Since $\gamma_\l>0$ and $\beta>0$, 
$\max I^+$ is well defined. 
 Because of Assumption~\ref{assume:gal}, $C_L$ and consequently $\mu_L$ are continuous.
  As we consider a symmetric profile, we have $\mu_L(\gamma_\l)=\frac{\sigma_\l}{4}C_L(0)=0$ whereas $\mu_G(\gamma_\l)>0$. Because of Inequality~\eqref{assum:mu}, the existence of solution of~\eqref{eq:mu_mu} in $[\gamma_{\monlambda},\max I^+]$ then follows from Intermediate Value Theorem. 
 Since $\mu_G\geq $ on $[\gamma_{\monlambda},\max I^+]$, the resulting lift is positive.

Suppose finally that $\max I^+=\theta_\l$. Assumption~\ref{assume:gal} guarantees that $\mu_L$ is positive on $[\gamma_\l,\theta_\l]$. Since $\mu_G(\theta_\l)=0$, the last assertion follows.
\end{proof}

In the case where $\mu_L$ is supplementary assumed to be increasing on $[\gamma_\l,\beta+\gamma_\l]$, then, the solution defined in Theorem~\ref{th:exist_standard} is unique.




\subsection{Corrected model}
We now consider the corrected model defined by~(\ref{Correc:Eq1}--\ref{Correc:Eq3}), for a given value $a_c\in(0,1)$.
The algebraic manipulations performed in the previous section to get Theorem~\ref{lem:reform_standard} cannot be pushed as far as with the simplified model and the resulting formula still contain the unknown $a$.  Hence, before stating a reformulation of this model and an existence result, we need to clarify the dependence of $a$ on the variable $\varphi$. Again, we express our assumptions in terms of $\mu_L^c$ and $\mu_D^c$, but the translation in terms of $C_L$, $C_D$, $\sigma_\l$ and $F_\l(\varphi)$ is straightforward. 

In all this section, we suppose that $0\in I$, i.e. $|\gamma_\l|\leq \beta $, which means in particular that  
 $I^+=(0,\min\{\theta_\l,\beta + \gamma_{\monlambda}\}]$.




\begin{lemma}\label{lem:implicit} 

Assume that Assumption~\ref{assume:gal} holds and define, for $\varphi\in I^+$ 
\begin{equation}\label{assum:muD}
g(\varphi):=\tan^{-1}\varphi\tan(\theta_\l-\varphi) + \frac{\mu_D^c(\varphi)}{\sin\varphi}\left(1+\tan^{-1}\varphi\tan(\theta_\l-\varphi)\right).
\end{equation}
Let $\psi$ be one of the functions given in Table~\ref{Table_correc}, with $F_\l(\varphi)=1$ in the case of Glauert's empirical correction. Then, the equation
\begin{equation}\label{eq:a=f}
\frac{a}{1-a}+ \frac{\sin\theta_\l\sin\varphi}{\cos(\theta_\l-\varphi)}\frac{\psi\left( (a-a_c)_+\right)}{(1-a)^2}=g(\varphi)
\end{equation}
defines a continuous  mapping $\tau: \varphi\in I^+\mapsto a\in[0,1)$. 

Moreover, if $g$ is decreasing and $\mu^c_D$ differentiable on  $I^+$, then $\tau$ is decreasing and differentiable for all $\varphi \in I^+$ with a possible exception of one point~$\varphi_c$, \ComM{which satisfies $\tau(\varphi_c)=a_c$}.
\end{lemma}

Because of Assumption~\ref{assume:gal}, the function $\mu^c_D$ is non-negative and defined for all angles in concrete cases so that $g$ is well defined on $I^+$.
The only obstruction for $g$ to be decreasing would come \CCom{from this term. Indeed,} $C_D$ 
 \rCCCom{ typically increases as angle increases from zero,
 hence our assumption 
in Section~\ref{Slow_inc_Drag}}. But for usual profiles, its variations are negligible
when compared to the other (decreasing) terms in~\eqref{assum:muD}.

\begin{proof}
For simplicity of notation, let us rewrite~\eqref{eq:a=f} under the form
\begin{equation}\label{eq:a=f_simpli}
u(a)+v(\varphi)w(a)=g(\varphi),
\end{equation}
with $u(a)		:=\frac{a}{1-a}$, $v(\varphi):=\frac{\sin\theta_\l\sin\varphi}{\cos(\theta_\l-\varphi)}$, $w(a)	:=\frac{\psi\left( (a-a_c)_+\right)}{(1-a)^2}$.
Let us first consider the left-hand side of~\eqref{eq:a=f_simpli}. We see that $u$ is positive and increasing on $[0,1)$ as well as $w$ for any function $\psi$ given in Table~\ref{Table_correc} (with $F_\l(\varphi)=1$ in the case of Glauert's empirical correction). In the same way, $v$ is positive and increasing on $(0,\theta_\l]$, hence on $I^+$. 
Fix now $\varphi\in I^+$, it is fairly easy to see that the mapping $a\in[0,1)\mapsto u(a)+v(\varphi)w(a)$ is continuous, strictly increasing, 
strictly positive and goes from $0$ to $+\infty$. 
Since $g$ is bounded and assumed to be positive on $I^+$, 
there exists an only $a$ in $[0,1)$ such that~\eqref{eq:a=f_simpli} holds. Hence the existence of the mapping $\tau$.

Suppose now that $g$ is decreasing and $\mu^c_D$ differentiable.
If we set aside the function $w$ in the point $a=a_c$, all the functions involved in~\eqref{eq:a=f_simpli} are differentiable. Consider $\varphi\in I^+$, such that $\tau(\varphi)\neq a_c$. 
 Differentiating~\eqref{eq:a=f_simpli} with respect to $\varphi$ gives $\tau'(\varphi)=\frac{g'(\varphi)-v'(\varphi)w\left(\tau(\varphi)\right)}{u'\left(\tau(\varphi)\right)+v(\varphi)w'\left(\tau(\varphi)\right)}$. Combining the fact that $g$ is decreasing with the above properties of $v, w, u$ and their derivatives implies that $\tau'(\varphi)\leq 0$. As a consequence, the mapping $\tau$ is decreasing and differentiable either on 
 $I^+$, or on $I^+-\{\varphi_c\}$ where $\varphi_c$ is the unique value in $I^+$ such that $\tau(\varphi_c)=a_c$. The result follows.
\end{proof}

\begin{remark}\label{rem=tau} The quantity $a=\tau(\varphi)$ can generally be computed explicitly provided that the function $\psi$ is specified analytically as, e.g., in Table~\ref{Table_correc}. In these cases, the computation consists in solving a low order polynomial equation (in $a$).
\end{remark}
We can now state a result similar to Theorem~\ref{lem:reform_standard} in the case of the corrected model.

\begin{theorem}
\label{lem:reform_correc}  
Let Assumption~\ref{assume:gal} 
hold and $\psi$ be one of the functions given in Table~\ref{Table_correc}, with $F_\l(\varphi)=1$ in the case of Glauert's empirical correction.
Suppose also that $(\varphi,a,a')\in 
 \rComM{I^+-\{\frac{\pi}2\}\times [0,1)\times \R-\{-1\}}
$ satisfies~(\ref{Correc:Eq1}--\ref{Correc:Eq3}).
Then $\varphi$ satisfies 
\begin{equation}\label{eq:corr_mu}
\mu_L^c(\varphi)-\tan(\theta_\l-\varphi)\mu_D^c(\varphi)=\mu_G^c(\varphi),
\end{equation}
where 
\begin{equation}\label{defmuGc}
\mu_G^c(\varphi):=\mu_G(\varphi) + \dfrac{\cos\theta_\l{\sin^2\varphi}}{\cos(\theta_\l-\varphi)}\dfrac{\psi\left( (\tau(\varphi)-a_c)_+\right)}{(1-\tau(\varphi))^2}.
\end{equation}
Reciprocally, suppose that $\varphi \in I^+-\{\frac{\pi}2\}$ satisfies~\eqref{eq:corr_mu} and define $a$ and $a'$ by
\begin{align}
a =& \tau(\varphi),\label{eq:corr_a}\\ 
a'=&\dfrac{1-\tau(\varphi)}{\l \sin^2\varphi} (\mu^c_L\sin\varphi - \mu^c_D \cos\varphi).\label{eq:corr_ap}
\end{align}
Then
$(\varphi,a,a')\in 
 \rComM{I^+-\{\frac{\pi}2\}\times [0,1)\times \R-\{-1\}}
$ satisfies~(\ref{Correc:Eq1}--\ref{Correc:Eq3}).
\end{theorem}
We refer to Section~\ref{numintro} for concrete examples of graphs of $\mu^c_G$ 
 and $\varphi\mapsto \mu_L^c(\varphi)-\tan(\theta_\l-\varphi)\mu_D^c(\varphi)$. 
As was the case with the simplified model, the value $\varphi=\frac\pi2$ is excluded only for the technical reason that~\eqref{Correc:Eq1} is not defined for this angle. 

\begin{proof}
Thanks to Lemma~\ref{lem:implicit}, $\tau$ is well defined on $I^+$.
Let $(\varphi,a,a')\in I^+-\{\frac{\pi}2\}\times [0,1)\times \R^+$ satisfying~(\ref{Correc:Eq1}--\ref{Correc:Eq3}). Because of~\eqref{Correc:Eq1}, 
$\tan^{-1}\varphi =\lambda_{\CCCom{r}}\frac{1+a'}{1-a} =\lambda_{\CCCom{r}}(1+\frac{a}{1-a}) + \lambda_{\CCCom{r}}\frac{a'}{1-a}$.
In this equation, the terms $\frac{a}{1-a}$ and $\lambda_{\CCCom{r}}\frac{a'}{1-a}$ can be eliminated thanks to~\eqref{Correc:Eq2} and~\eqref{Correc:Eq3}, respectively. After some algebraic manipulations, we end up with 
 $\mu_L^c=(\sin\varphi+ \mu_D^c)\tan(\theta_\l-\varphi) + \dfrac{\cos\theta_\l{\sin^2\varphi}}{\cos(\theta_\l-\varphi)}\dfrac{\psi\left( (a-a_c)_+\right)}{(1-a)^2}$,
 hence~\eqref{eq:corr_mu}. Using this formula to eliminate $\mu_L^c$ in~\eqref{Correc:Eq2} gives~\eqref{eq:a=f}. Consequently, Lemma~\ref{lem:implicit} implies that $a$ and $\varphi$  satisfy~\eqref{eq:corr_a}. Finally,~\eqref{eq:corr_ap} follows from~\eqref{eq:corr_a} and~\eqref{Correc:Eq3}.

Suppose now that \rComM{$\varphi\in I^+-\{\frac{\pi}2\}$ satisfies~\eqref{eq:corr_mu}, and let $(a,a')$ defined by (\ref{eq:corr_a}--\ref{eq:corr_ap}), meaning that $a\in[0,1)$}. Replacing $\tau(\varphi)$ by $a$ in~\eqref{eq:corr_ap} gives immediately~\eqref{Correc:Eq3}. Combining~\eqref{eq:corr_a} with the definition of $\mu_G^c$ give\rComM{s}
~\eqref{Correc:Eq2}. Finally,~\eqref{Correc:Eq1} is obtained by introducing $a$ and $a'$ into~\eqref{eq:corr_mu}, \rComM{where $\varphi\neq \frac{\pi}2$ implies that $a'\neq -1$}.
\end{proof}

As in the simplified model, Glauert's model boils down to a scalar equation in $\varphi$. 
 However, 
formulation~\eqref{eq:corr_mu} does not completely decompose the terms into a local part 
and macroscopic modeling part:  
 much as the left-hand side of~\eqref{eq:corr_mu} still only relies on the turbine 
the right-hand side now also depends on it via $\tau$, since~\eqref{eq:a=f} includes $\mu_D^c$.
Before going further, let us give more details about the behavior of $\tau$ in $\varphi=0$.

\begin{lemma}\label{asympt}
Let Assumption~\ref{assume:gal} 
hold, $\mu_D^c(0)\neq 0$ and $\psi$ be one of the functions given in Table~\ref{Table_correc}, with $F_\l(\varphi)=1$ in the case of Glauert's empirical correction. 
 Then
 $\tau(\varphi)=1-\sqrt{\frac{\psi(1-a_c)}{\mu_D^c(0)}}\varphi^{3/2}+o_{\varphi=0}(\varphi^{3/2})$.
\end{lemma}
\begin{proof} 
Thanks to Lemma~\ref{lem:implicit}, $\tau$ is well defined on $I^+$.
Let us first prove that $\Lim{\varphi\rightarrow 0^+}\tau(\varphi)=1^-$. From~\eqref{assum:muD}, we see that $\Lim{\varphi\rightarrow 0^+}g(\varphi)=+\infty$.  Given $\varphi \in \rComM{I^+}$, we have $a=\tau(\varphi)\in [0,1)$ and $ 1-\frac{\cos\theta_\l\cos\varphi}{\cos(\theta_\l-\varphi)} =\frac{\sin\theta_\l\sin\varphi}{\cos(\theta_\l-\varphi)} \geq 0$, so that all the terms of the left-hand side  of~\eqref{eq:a=f} are positive. As a consequence, the only possibility for the sum of these terms to go to $+\infty$ is that $\Lim{\varphi\rightarrow 0^+}\tau(\varphi)=1^-$.

Define now $\nu(\varphi)=1-\tau(\varphi)$. Expanding~\eqref{eq:a=f} in a neighborhood $\varphi=0^+$, we get
\[
\frac{1}{\nu(\varphi)} - 1 +\left(\tan \theta_\l . \varphi 
+ o_{\varphi=0}(\varphi) \right)\frac{\psi\left( (1-a_c)_+\right) + o_{\varphi=0}(1)}{\nu(\varphi)^2}
=\tan\theta_\l \frac{\mu_D^c(0)}{\varphi^2} + o_{\varphi=0}(\frac 1{\varphi^2}),
\]
which implies:
\begin{align}
\frac{\nu^2(\varphi)}{\varphi^3}\left( \frac{\varphi^2}{\nu(\varphi)}- \varphi^2 - (\tan\theta_\l .	\mu_D^c(0) + o_{\varphi=0}(1))\right)&=-\tan\theta_\l\psi\left( (1-a_c)_+\right).\label{absurd}
\end{align}

Let $(\varphi_n)_{n\in\N}$ a sequence satisfying $\Lim{n} \varphi_n=0^+$, so that $\Lim{n} \nu( \varphi_n)=0^+$. Suppose that $\Lim{n} \frac{\nu^2( \varphi_n)}{\varphi^3_n}=+\infty$. Since $\frac{\varphi_n^2}{\nu(\varphi_n)}= \left(\frac{\varphi_n^3}{\nu^2(\varphi_n)}\right)^{2/3} \nu^{1/3}(\varphi_n) $,  
{this sequence goes to zero.}
Back to~\eqref{absurd}, we find a contradiction since the left-hand side goes to $+\infty$ whereas the right-hand side is constant. It follows that, up to a subsequence, we can assume that $\Lim{n} \frac{\nu^2( \varphi_n)}{\varphi^3_n}=\ell$ for a certain $\ell$. Setting $\varphi=\varphi_n$ in~\eqref{absurd} and passing to the limit $n\rightarrow +\infty$, we obtain that $\ell=\frac{\psi\left( (1-a_c)_+\right)}{\mu_D^c(0)}$. The result follows.
\end{proof}

\begin{remark}
If $\mu_D^c(0)=0$, we obtain 
 $\tau(\varphi)=1-\sqrt{\psi(1-a_c) \varphi}+o_{\varphi=0}(\varphi^{1/2})$.
\end{remark}

The quantity $\mu_D^c(0)=\frac{\sigma_\l}{4} C_D(-\gamma_{\monlambda})$ has no specific physical meaning in the applications.  We have introduced it as a constant (that can be expressed explicitly), for simplicity of presentation. \rCCCom{As a matter of fact,  $\varphi=0$ 
 is a specific angle 
from the macroscopic point of view}, as appears when considering $\mu_G$ (that cancels in $0$) and $\mu_G^c$, see the proof of the next result. 

We are now in a position to give an existence result about the corrected model.
\begin{corollary}[of Theorem~\ref{lem:reform_correc}] \label{cor:exist_corrected}
Suppose that Assumption~\ref{assume:gal} 
 holds and that
\begin{equation}\label{assum_correc_angle1}
\mu^c_G(\max I^+)\leq \mu^c_L(\max I^+)-\tan(\theta_\l-\max I^+)\mu^c_D(\max I^+).
\end{equation}
Then~\eqref{eq:corr_mu} admits 
a solution in 
$I^+$ 
corresponding to a positive lift.
Moreover, if $g$ is decreasing, $\max I^+=\theta_\l$ and $\rComM{\varphi_c} <\theta_\l$, where $\rComM{\varphi_c}$ is defined in Lemma~\ref{lem:implicit}, then~\eqref{assum_correc_angle1} is necessarily satisfied. 
\end{corollary}
\begin{proof}
Because of Lemma~\ref{asympt} and Definition~\eqref{defmuGc} of $\mu^c_G$, we get 
$\mu^c_G(\varphi) \approx_{\varphi\rightarrow 0^+}  \frac{\mu_D^c(0)}{\varphi},$
so that 
 $\Lim{\varphi\rightarrow 0^+} \mu^c_G(\varphi) = +\infty$.
This implies that there exists a small enough $\varphi_0>0$ such that $\mu^c_G(\varphi_0) \geq \mu_L^c(\varphi_0)-\tan(\theta_\l-\varphi_0)\mu_D^c(\varphi_0)$.
Because of the assumption~\eqref{assum_correc_angle1}
, the existence of a solution of~\eqref{eq:corr_mu} then follows from Intermediate Value Theorem.  The positivity of $\mu^c_G$ on $I^+$ implies that the resulting lift is positive.

Suppose now that $g$ is decreasing, $\max I^+=\theta_\l$ and $\rComM{\varphi_c} < \theta_\l$. \rComM{Lemma~\ref{lem:implicit}} implies that $\tau$ is decreasing   
so that the correction associated with $\psi$ is not anymore active on $[\rComM{\varphi_c},\theta_\l)$. 
We then have
$\mu^c_G(\max I^+)=\mu_G(\max I^+)=\mu_G(\theta_{\monlambda})=0 $ whereas $  \mu^c_L(\max I^+) -\tan(\theta_\l-\max I^+)\mu_D^c(\max I^+) =\mu^c_L(\max I^+) \geq 0$.
Hence~\eqref{assum_correc_angle1}. \end{proof}


Unlike the simplified model, no condition on $\gamma_\l$ or $\mu_L^c(\gamma_\l)$ \rCCCom{is assumed in Corollary~\ref{cor:exist_corrected}}  
, but the alternative 
assumption $0\in I$ is required. This makes the corrected model much better posed than its simplified version.

\begin{remark}
In the case $\mu^c_D(0)=0$, similar reasoning gives $\mu^c_G(\varphi) \approx_{\varphi\rightarrow 0^+}(1+\tan\theta_\l)\varphi$. 
As a consequence, $\mu^c_G(0)=0$, so that, as in the simplified model, one needs an assumption about, e.g., $\mu^c_L(\gamma_\l)$ to get an existence result similar to Corollary~\ref{th:exist_standard}.  
\end{remark}

\subsection{Multiple solutions}\label{multiple}
The results of the previous sections can be completed by some additional remarks about cases of multiple solution. 
More precisely, these cases can be sorted into three independent categories:
\begin{enumerate}
\item {Multiple solutions in the simplified model:} since 
$\Lim{\varphi\rightarrow \theta_\l \pm \pi/2}\mu_G(\varphi)=-\infty$, 
there shall be two intersections between the graphs of $\mu_G$ and $\mu_L$, e.g. in the case where $\mu_L$ is affine on a large enough interval, $C_L(0)=0$ and $\gamma_\l\in (0,\theta_\l]$. In this case, one of the two roots gives rise to a negative lift.

\item {Multiple solutions caused by stall:} as mentioned in Section~\ref{locvar}, the stall phenomenon is generally associated with a sudden decrease in $C_L$. It follows that if the stall angle $\alpha_s$ satisfies $\alpha_s+\gamma_\l \in I$, the graph of $\mu_L$ shall cross the graph of $\mu_G$ at an angle in $\varphi\geq \alpha_s +\gamma_\l$. This fact is reported in~\cite[p.129]{Manwell}: 

\textit{``In the stall region [...] there may be multiple solution for $C_L$. Each of these
solutions is possible. The correct solution should be that which maintains the continuity of the
angle of attack along the blade span.''}

\item {Multiple solutions in the corrected model:} since
$\Lim{\varphi\rightarrow 0^+} \mu^c_G(\varphi) = +\infty$,  
the graph of $\mu_G^c$ may no longer be concave on $I^+$ 
when a correction for large values of $a$ is active. 
Hence possible multiple solution, e.g. in the case $\mu_L$ is affine. 

\end{enumerate}
Concrete examples of these three types of multiple solution are given in Section~\ref{numres}.

\section{Solution algorithms}\label{Sec:algo}
To solve numerically Glauert's model, a specific fixed-point \CCCom{approach} is often highlighted in the literature. 
In this section, we recall its main features and introduce more efficient procedures.



\subsection{\CCCom{Standard fixed-point procedure} }\label{sec:algo_stand}
Solving the simplified or the corrected model is usually done by a dedicated fixed-point procedure \ComM{that comes in two versions}, see~\cite{Branlard,Hansen,Manwell,Lubitz,Sorensen}\footnote{Version 1 is mentionned in~\cite{Branlard} and~\cite{Manwell}.}
 or the early presentation  in~\cite[p.47]{Wilson}. 
\begin{algorithm}
\caption{Solving BEM system, \CCCom{Standard fixed-point (versions 1 and 2)}}
\begin{algorithmic}
\STATE{\textbf{Input:} $\textrm{Tol}>0$, $\alpha\mapsto C_L(\alpha)$, $\alpha\mapsto C_D(\alpha)$, $\lambda_{\CCCom{r}}$, $\gamma_{\l}$, $\sigma_{\l}$, $F_{\l}$, $x\mapsto \psi(x)$.}\\
\STATE{\textbf{Initial guess:}  $a,a'$.}\\
\STATE{\textbf{Output:}  $a, a', \varphi$.}\\

\STATE Set $err \vcentcolon= \textrm{Tol}+1.$\;
\WHILE{$err>\textrm{Tol}$}{
\item (1) Set 
$\varphi \vcentcolon= {\rm atan}\left(\frac{1-a}{\lambda_{\CCCom{r}}(1+a')}\right).$
\;
\item (2) Set $a$ as the solution of~\eqref{Correc:Eq2}.\;
\item (3) \ComM{Set $a'$ as the solution of~\eqref{Correc:Eq3} (version 1) or of~\eqref{Correc:Eq3_alter} (version 2)}
\item (4) Set $err \vcentcolon= \left| \tan\varphi - \frac{1-a}{\lambda_{\CCCom{r}}(1+a')}\right|$.
}
\ENDWHILE
\end{algorithmic}
\label{algo_corrected}
\end{algorithm}
This procedure is given in Algorithm~\ref{algo_corrected}, where
 the stopping criterion is arbitrary and usually not mentioned in monographs. 
The convergence of these algorithms is problematic. Instabilities are often observed in practice, as reported, e.g., in~\cite{Lubitz}:

\textit{``Note that this set of equations must be solved simultaneously, and in practice, numerical
instability can occur.''}

\textit{``When local angle of attack is around the stall point, or becomes negative, getting the BEM
code to converge can become difficult.''}\\
We also refer to~\cite{Maniaci} for a specific study of some convergence issues. The analysis of the algorithm is tedious ; we refer to Appendix for an example of setting where the convergence of \ComM{Version 1} is guaranteed. 

\subsection{Optimized fixed-point procedures}
\label{algo_new}
\CCCom{To cure the convergence issues of observed when using
 the standard fixed-point procedure, various alternative fixed-point procedures have been proposed in the last decade. 
A Newton-Raphson procedure have been studied numerically in~\cite{Williams}. As usual with Newton's iteration, this method outperforms the standard fixed-point in case of convergence. However, this approach fails to converge in some regimes, which can be described numerically in terms of $\lambda_{\CCCom{r}}$ and $\sigma$. Sun \textit{et al.} proposed to modified the standard fixed-point by introducing a relaxation term, i.e., replace the step (1) in Algorithm~\ref{algo_corrected} 
 by 
$\varphi \vcentcolon= (1-w)\varphi+ w\ {\rm atan}\left(\frac{1-a}{\lambda_{\CCCom{r}}(1+a')}\right)$, 
for $w\in(0,1]$. In case of non-convergence, $w$ is divided by 2. This procedure is tested numerically in~\cite{Sun}.
Thanks to our new formulation, we propose now alternative fixed-point procedures whose convergence can be guaranteed in some cases.
}
\subsubsection{General formulation}
\label{sec:new fixed point}
In view of~\eqref{eq:corr_mu}, we consider now optimized fixed-point procedures
based on the iteration
\begin{align}
\varphi^{k+1}=&f(\varphi^k),\label{eq:rec_new_c}
\end{align}
with  
 $f(\varphi)=\varphi + \rho(\varphi)Res(\varphi)$ 
 , where $Res(\varphi):=\mu^c_G(\varphi) - \mu^c_L(\varphi) + \tan(\theta_\l-\varphi)\mu^c_D(\varphi)$ and $\rho(\varphi)>0$ is a given \Com{relaxation coefficient}. The procedure is summarized in Algorithm~\ref{algo_new_algo}.
\begin{algorithm}
\caption{Solving BEM system, \CCCom{Optimized fixed-point procedure}}
\label{algo_new_algo}
\begin{algorithmic}

\STATE{\textbf{Input:} $\textrm{Tol}>0$, $\alpha\mapsto C_L(\alpha)$, $\alpha\mapsto C_D(\alpha)$, $\lambda_{\CCCom{r}}$, $\gamma_{\l}$, $\sigma_{\l}$, $F_{\l}$, $x\mapsto \psi(x)$.}\\
\STATE{\textbf{Initial guess:}  $\varphi$.}\\
\STATE{\textbf{Output:}  $\varphi$.}\\

\STATE Set $err \vcentcolon= \textrm{Tol}+1.$\;
\WHILE{$err>\textrm{Tol}$}{
\item Compute $a:=\tau(\varphi)$, i.e. the solution of~\eqref{eq:a=f}.\;
\item Compute $\mu^c_G(\varphi)$, using~\eqref{defmuGc} and $f(\varphi)$.\;
\item Set $err \vcentcolon= \left| f(\varphi) - \varphi\right|$.\;
\item Set $\varphi=f(\varphi)$.
}
\ENDWHILE
\end{algorithmic}
\end{algorithm}
\ComM{The parameter of $\rho(\varphi)$ can be optimized to obtain a robust version of this procedure or a Newton procedure.
\subsubsection{Robust version}
Defining 
\begin{equation}\label{def:rho_c}
\rho(\varphi)=\rho_\varepsilon(\varphi):=\dfrac{\varepsilon}{ \max\left\{0,-\mu^{\ComM{c}}_G{}'(\varphi)\right\} 
+ 
\max\limits_{I^+}\mu^{c}_L{}'   
+\left(1+\tan^2\theta_\l\right)\mu^{c}_D(\varphi) }
\end{equation}
leads to a robust algorithm, whose convergence is guaranteed when $\psi=0$ (or when $a$ remains below $a_c$).}
%
%
%
%
%
%
\begin{theorem} \label{th:cv}
Suppose that $\max I^+=\theta_\l$, $\psi=0$, and that Assumption~\ref{assume:gal} holds. Assume also that $\mu^{c}_L{}$ and $\mu^{c}_D{}$ are continuously differentiable on $ I^+$, with $\mu^{c}_L{}$ and $\mu^{c}_D{}$ non-decreasing and that~(\ref{Standard:Eq1}--\ref{Standard:Eq3}) admit a solution in $ I^+ $. 
Given $\varepsilon\in (0,1)$ 
the sequence $(\varphi^{k})_{k\in\N}$ defined by~\eqref{eq:rec_new_c} with $\rho(\varphi)=\rho_\varepsilon(\varphi)$ defined in~\eqref{def:rho_c}
 and the initial value
 $\varphi^{0}=\theta_\l \label{eq:rec_new_c_0}$ 
 converges to $\varphi^\star$, the largest solution of~\eqref{eq:corr_mu} in $I^+$.
\end{theorem} 

\begin{proof}
The assumption on $\mu^{c}_L{}$ 
guarantees that the denominator in~\eqref{def:rho_c} \ComM{is strictly positive on $I^+$} 
so that $\rho_\varepsilon(\varphi)$ is well-defined and positive \ComM{on this interval}.  
\ComM{Since $\mu^{c}_G=\mu_G$ is concave}, $\varphi\mapsto\rho_\varepsilon(\varphi)$ is decreasing on $I^+$.
Since $Res(\varphi)\leq 0$ 
 on $[\varphi^\star,\theta_\l]$, we get:
\begin{align*}
f{}'(\varphi)	=&1 + \rho_\varepsilon(\varphi)\left(\mu_G'(\varphi) - \mu^{c}_L{}'(\varphi) -(1+ \tan^2(\theta_\l-\varphi))\mu^{c }_D(\varphi) + \tan(\theta_\l-\varphi)\mu^{c}_D{}'(\varphi)\right) \\ &+ \rho_\varepsilon{}'(\varphi)\ComM{Res(\varphi)}\\
		\geq& 1 - \rho_\varepsilon(\varphi) \left(
\max\left\{0,-\mu_G'(\varphi)\right\} 
+ \max\limits_{I^+}\mu^{c}_L{}'   +\left(1+\tan^2\theta_\l\right) \mu^{c}_D(\varphi) \right)
		= 1-\varepsilon \geq 0,
\end{align*}
so that $f$ is increasing on $[\varphi^\star,\theta_\l]$. Since $\mu^{c}_L{}$ is non-decreasing, we have
\begin{align*}
f(\theta_\l)	&=\theta_\l + \rho_\varepsilon(\theta_\l)\left(\mu_G(\theta_\l) - \mu^c_L(\theta_\l) \right)
=\theta_\l - \rho_\varepsilon(\theta_\l)\mu^c_L(\theta_\l) \leq \theta_\l.
\end{align*}
These results and $f(\varphi^\star)=\varphi^\star$, imply that $f([\varphi^\star,\theta_\l])\subset [\varphi^\star,\theta_\l]$. Since $\varphi^{0}=\theta_\l$, $(\varphi^{k})_{k\in\N}$ is bounded and decreasing, hence converges.  The result follows.
\end{proof}

In some cases, we can estimate the rate of convergence of $(\varphi^{k})_{k\in\N}$.

\begin{theorem}
In addition to the assumptions of Theorem~\ref{th:cv}, suppose that
\begin{equation}\label{eq:CVbound}
 \tan\theta_\l ( 1+\max\limits_{I^+}\mu^{c}_D{}') < \min\limits_{I^+}\mu^{c}_L{}'.
\end{equation}
Then 
the sequence $(\varphi^{k})_{k\in\N}$ defined by~\eqref{eq:rec_new_c} satisfies
\[|\varphi^{k}-\varphi^\star|\leq \left(1 - \frac{\min\limits_{I^+}\mu^{c}_L{}' - \tan\theta_\l ( 1+\max\limits_{I^+}\mu^{c}_D{}') }{  \max\limits_{I^+}\mu^{c}_L{}' +\sin\theta_\l + (1+\tan^2\theta_\l )\mu^{c}_D(\theta_\l) }\right)^k |\theta_\l-\rho_\varepsilon(\theta_\l)\mu^c_L(\theta_\l)|.\]
\end{theorem}

\begin{proof}
For we have already shown in the previous proof that $f{}'(\varphi)\geq 0$ on $[\varphi^\star,\theta_\l]$, it remains to determine an upper bound for $f{}'(\varphi)$. To do this, we use the bound~\eqref{eq:CVbound} and $\mu_G'(\varphi)\leq \mu_G'(0)=\tan\theta_\l$ to get:
\begin{align*}
f{}'(\varphi)	&=1 + \rho_\varepsilon\left(\mu_G'(\varphi) - \mu^{c}_L{}'(\varphi) -(1+ \tan^2(\theta_\l-\varphi))\mu^{c }_D(\varphi) + \tan(\theta_\l-\varphi)\mu_D^{c}{}'(\varphi)\right)\\
		&\leq 1 - \frac{\min\limits_{I^+}\mu^{c}_L{}' - \tan\theta_\l   ( 1+\max\limits_{I^+}\mu^{c}_D{}') }
			       {\max\limits_{I^+}\mu^{c}_L{}' + \sin\theta_\l + (1+\tan^2\theta_\l )\mu^{c}_D(\theta_\l) },
\end{align*}
where we have used $\max\left\{0,-\mu_G'(\varphi)\right\} \geq -\mu_G'(\theta_\l)=\sin\theta_\l$ to bound $\rho_\varepsilon $ from below.
The result \rComM{is} then obtained by induction. 
\end{proof}

\ComM{\subsubsection{A Newton version}
One can actually obtain quadratic convergence, i.e.  $|\varphi^k-\varphi^\star|\leq \delta |\varphi^0-\varphi^\star|^{2^k}$ for some $\delta>0$ by using a Newton procedure, i.e., setting
\[\rho(\varphi):= -\frac 1 {\mu^{c}_G{}'(\varphi) - \mu^{c}_L{}'(\varphi) -(1+ \tan^2(\theta_\l-\varphi))\mu^{c }_D(\varphi) + \tan(\theta_\l-\varphi)\mu_D^{c}{}'(\varphi)}.\]
and by choosing $\varphi^0$ close enough to $\varphi^\star$. 
\Com{In this formula, the functions $\mu^{c}_L$ and $\mu^{c}_D$ are usually only known experimentally, i.e. pointwise. In practice, splines or polynomial interpolation \rCCCom{are} used to evaluate $C_L$ and $C_D$ for any arbitrary angle. In this way, the derivatives $\mu^{c}_L{}'$ and $\mu^{c}_D{}'$ (of the extension) can be obtained without any supplementary computational cost. The term $\mu^{c}_G{}'(\varphi)$ can also be computed without significant additional cost as soon as $\mu^{c}_G{}(\varphi)$ has been computed, see Remark~\ref{rem=tau}.}




}

\CCCom{\subsection{Root-finding algorithms}\label{Sec:RF}
Reducing~(\ref{Correc:Eq1}--\ref{Correc:Eq3}) to a one dimensional equation allows \rCCCom{the application of} usual root-finding algorithms such as bisection.
In this way, Ning~\cite{Ning} used Brent's procedure~\cite{Brent} to solve~\eqref{eqNing}.
 Using~\eqref{eq:mu_mu}, a new root-finding approach consists in applying Brent's procedure to the equation $Res(\varphi)=0$. If 
 a correction for high values of $a$ is considered, then the framework of Corollary~\ref{cor:exist_corrected} implies that there exists a solution of the corrected model in $I^+$. In this case, the convergence of bisection algorithm or Brent's procedure is guaranteed.}
Moreover, the solution found in the case $\psi=0$, e.g., by Algorithm~\ref{algo_new_algo}, can be used to bracket the solution 
 in a finer way  than $I^+$.
\begin{lemma} Keep the assumptions of Corollary~\ref{cor:exist_corrected}, and denote by $\varphi_0$ a solution~\eqref{eq:corr_mu} where $\psi=0$. 
Then~\eqref{eq:corr_mu} admits a solution in $(\varphi_0,\min\{\theta_\l,\beta + \gamma_{\monlambda}\}]$ corresponding to a positive lift.
\end{lemma}
\begin{proof}
Since $\psi=0$,~\eqref{eq:corr_mu} implies 
 $Res(\varphi_0)=-\dfrac{\cos\theta_\l{\sin^2\varphi_0}}{\cos(\theta_\l-\varphi_0)}\dfrac{\psi\left( (\tau(\varphi_0)-a_c)_+\right)}{(1-\tau(\varphi_0))^2}$, so that $Res(\varphi_0)\leq 0$.
The Intermediate Value Theorem and~\eqref{assum_correc_angle1} give the result.
\end{proof}

\section{Optimization}\label{Sec:opti}
%
%
%
%

The BEM model does not only aim to evaluate the efficiency of a given geometry, but also provides a framework to design rotors, that is, to select high-performance parameters $\gamma_{\l}$ and $c_{\l}$. In this way, monographs often consider a specific maximization procedure of a functional $C_p$, called {\it power coefficient} (\cite[p.129, (3.90a)]{Manwell}, \cite[p.328]{Glauert3})
, which corresponds to the ratio between the received and the captured energy. This quantity is usually defined by
\rComM{
$ C_p(\gamma_{\l},c_{\l},\varphi,a,a')
	=\int_{\l_{\min}}^{\l_{\max}} J_\l(\gamma_{\l},c_{\l},\varphi,a,a') d\l,
$ 
}
 where the elementary contribu\-tion $J_\l$ reads:
\rComM{ \begin{equation}\label{power-coef}
J_\l(\gamma_{\l},c_{\l},\varphi,a,a'):=\frac{8F_\l(\varphi)\l^3 }{\l_{\max}^2} a'(1-a)\left( 1-\frac{C_D(\varphi-\gamma_\l)}{C_L(\varphi-\gamma_\l)}\tan^{-1}\varphi\right).
\end{equation}
}
in which the variables satisfy the constraints~(\ref{Correc:Eq1}--\ref{Correc:Eq3}). The drag coefficient $C_D$ is consequently taken into account (though partly neglected in the reasoning, as explained hereafter)  as well as the \rCCCom{tip} loss correction. On the contrary, no correction related to high values of $a$ 
is considered. This motivates the description of an optimization algorithm for the full corrected model in Section~\ref{sec:grad}. In any case, the contributions are independent. As a consequence, we focus on the optimization problem associated with one element, i.e., we fix the value of $\l$ and optimize $J_\l$.



\subsection{Simplified model and usual optimum approximation}\label{sec:design}
The usual optimization procedure is described in, e.g.,~\cite[p.131-137]{Manwell}. For the sake of completeness, we recall it in the case where $F_\l=1$.

Considering independently each $\l$ on a discretization grid associated with the interval $[\l_{\min},\l_{\max}]$ and the corresponding functional $J_\l(\gamma_{\l},c_{\l},\varphi,a,a')$, the procedure starts by determining an angle $\overline{\alpha}$ which minimizes the ratio $\frac{C_D(\alpha)}{C_L(\alpha)}$. In the following steps, the coefficient $C_D$ is neglected: not only the factor $ 1-\frac{C_D(\varphi-\gamma_\l)}{C_L(\varphi-\gamma_\l)}\tan^{-1}(\varphi) $ is set to $1$ in~\eqref{power-coef}, but $C_D$ is also set to $0$ in the constraints, which correspond to the simplified model
~(\ref{Standard:Eq1}--\ref{Standard:Eq3}) afterwards.
Using Theorem~\ref{lem:reform_standard} to replace $\mu_L$ by $\mu_G$ in~(\ref{Standard:Eq2}--\ref{Standard:Eq3}), $a$, $a'$, and consequently $J_\l$ are expressed exclusively in terms of $\varphi$, namely 
 $a=1-\dfrac{\sin\varphi \cos(\theta_{\l}-\varphi)}{\sin \theta_{\l}}$, $a'= \dfrac{\sin\varphi \sin(\theta_{\l}-\varphi)}{\cos\theta_{\l}}$ and 
 $J_\l = \rComM{\frac{8\l^3 }{\l_{\max}^2}} \frac{ \sin^2\varphi\sin\left(2(\theta_{\l}-\varphi)\right)}{\sin 2\theta_\l}$.
As a consequence, it remains to optimize $\varphi \mapsto\sin^2\varphi \sin (2(\theta_{\l}-\varphi))$ on $[0,\theta_{\l}]$. \ComM{It is easily seen that} the maximum is attained at $\varphi^{*} = \frac{2}{3} \theta_{\l}$. Finally, 
 $\gamma_{\l}^{\ast} := \gamma_{\l}(\varphi^{\ast})$ and $c_{\l}^{\ast} := c_{\l}(\varphi^{\ast})$ can be computed from~\eqref{eq:def_alpha} and \eqref{eq:mu_mu}, which gives
\begin{equation}\label{opt_simp}
\gamma_{\l}^{\ast} := \varphi^{\ast} - \overline{\alpha},\;\;
c_{\l}^{\ast} := \dfrac{8\pi r\mu_G(\varphi^{\ast})}{BC_L(\overline{\alpha})}.
\end{equation}

\subsection{A gradient method for the corrected model}\label{sec:grad}
We now detail  an adjoint-based gradient method to tackle the optimization of $\gamma_\l$ and $c_\l$ in the framework of the corrected model. 
 Throughout this section, 
 $C_L'$ and $C_D'$ denote the derivatives of $C_L$ and $C_D$. We omit in the notation the dependence of $\mu^c_L$ and $\mu^c_D$ 
on $\varphi$ \rComM{and $c_\l$}.

We first recall how the introduction of Lagrange multipliers enables to compute the gradient of $J_\l$. 
Define the Lagrangian of Problem~\eqref{power-coef} by
\begin{align*}
{\cal L}_\l(\varphi,a,a', &\, p_1,p_2,p_3,c_\l,\gamma_\l)
= 
J_\l(\varphi,a,a',c_\l,\gamma_\l)-p_1\left(
\rComM{\tan\varphi-\frac{1-a}{\lambda_r(1+a')}}
 \right)\\
	&\quad -p_2\left( \frac{a }{1-a}-\frac{ 1}{\sin^2\varphi} (\mu^c_L\cos\varphi + \mu^c_D \sin\varphi)+\frac{\psi\left( (a-a_c)_+\right)}{(1-a)^2}\right)\\
	&\quad -p_3\left( \frac{a'}{1-a}-\frac{ 1}{\lambda_{\CCCom{r}} \sin^2\varphi} (\mu^c_L\sin\varphi - \mu^c_D \cos\varphi\right),
\end{align*}
where $p_1$, $p_2$ and $p_3$ are the Lagrange multipliers associated with the constraints~(\ref{Correc:Eq1}--\ref{Correc:Eq3}). The optimality system is obtained by canceling 
the partial derivatives of  ${\cal L}_\l$. Differentiating ${\cal L}_\l$ with respect to $p_1$, $p_2$ and $p_3$ and equating the resulting terms to zero gives the corrected model~(\ref{Correc:Eq1}--\ref{Correc:Eq3}), that can be solved using the algorithms presented in Section~\ref{Sec:algo}. Canceling the derivatives of ${\cal L}_\l$ with respect to $(\varphi,a,a')$ gives 
\begin{equation}\label{eq:p}
	M\cdot p = b,
\end{equation}
where $p\vcentcolon=( p_1 \ \ p_2\ \ p_3)^{\t}$ is the Lagrange multiplier vector, and
\begingroup
\begin{align*}
M:=&\begin{bmatrix}[1.6]
\frac 1{\cos^2\varphi} 
	&\rComM{v_\varphi
}
	&\rComM{w_\varphi
}
	\\
 \vspace{0.25cm}
-\frac1{\lambda_{\CCCom{r}}(1+a')} 
	& \frac{1+\psi'((a-a_c)_+)}{(1-a)^2} + \frac{2\psi((a-a_c)_+)}{(1-a)^3} 
	& \frac{a'}{(1-a)^2}
	\\
	\frac{1-a}{\lambda_{\CCCom{r}}(1+a')^2} 
	& 0 
	& \frac 1{1-a}
\end{bmatrix}
\\\rComM{
v_\varphi:=}&\rComM{\frac{2\cos\varphi(\mu^c_L\tan^{-1}\varphi +\mu^c_D)
-((\dpart{\mu_L^c}{\varphi}+\mu_D^c)\cos\varphi +(\dpart{\mu_D^c}{\varphi}+\mu_L^c)\sin\varphi)}{\sin^2\varphi}
}\\
\rComM{
w_\varphi:=}&\rComM{\frac{2\cos\varphi(\mu^c_L -\mu^c_D\tan^{-1}\varphi)
-((\mu_L^c-\dpart{\mu_D^c}{\varphi})\cos\varphi +(\mu_D^c+\dpart{\mu_L^c}{\varphi})\sin\varphi)}{\l\sin^2\varphi}
}\\
b:=&\frac{8F_\l(\varphi)\lambda_{\CCCom{r}}^3}{\lambda_{max}^2}
\begin{pmatrix}[1.5]
\rComM{
	\frac{a'(1-a)}{C_L(\varphi-\gamma_\l)}\left(\frac{  C_L'(\varphi-\gamma_\l) C_D(\varphi-\gamma_\l) -C_D'(\varphi-\gamma_\l) C_L(\varphi-\gamma_\l)}{C_L(\varphi-\gamma_\l)\tan\varphi} + \frac{C_D(\varphi-\gamma_\l)}{ \sin^2\varphi} \right)
}
	\\
	-a'\rComM{(1-\frac{C_D(\varphi-\gamma_\l)}{C_L(\varphi-\gamma_\l)}\tan^{-1}\varphi)}
	\\
	(1-a)\rComM{(1-\frac{C_D(\varphi-\gamma_\l)}{C_L(\varphi-\gamma_\l)}\tan^{-1}\varphi)}
\end{pmatrix}\\
&\quad +\frac{8F'_\l(\varphi)\lambda_{\CCCom{r}}^3}{\lambda_{max}^2}
\begin{pmatrix}[1.5]
	a'(1-a)\rComM{(1-\frac{C_D(\varphi-\gamma_\l)}{C_L(\varphi-\gamma_\l)}\tan^{-1}\varphi)}
	\\
	0
	\\
	0
\end{pmatrix}.
\end{align*}
\endgroup
Fix now the values of the pair $(\gamma_\l, c_\l)$ and set $\varphi, a, a', p$ as the corresponding solutions of~(\ref{Correc:Eq1}--\ref{Correc:Eq3}) and~\eqref{eq:p}, respectively. The gradient $\nabla J_\l(\gamma_\l, c_\l)$  
reads
\begin{equation}\label{eq:grad}
\nabla J_\l(\gamma_\l, c_\l) = 
	\left(
		\dpart{\mathcal{L}_\l}{\gamma_\l} \quad 
		\dpart{\mathcal{L}_\l}{c_\l}	
	\right)^{\t},
\end{equation}
where
\begin{align*}
\dpart{\mathcal{L}_\l}{\gamma_\l} 
	&=\rComM{\frac{8F_\l(\varphi)\l^3 }{\l_{\max}^2}} a'(1-a) \frac{C_D'(\varphi-\gamma_\l)C_L(\varphi-\gamma_\l)-C_L'(\varphi-\gamma_\l)C_D(\varphi-\gamma_\l)}{C_L^2(\varphi-\gamma_\l) \tan\varphi} \\
	&\quad - p_2\frac{1}{\sin^2\varphi}(\dpart{\mu^{\rComM{c}}_L}{\varphi}\cos\varphi+\dpart{\mu^{\rComM{c}}_D}{\varphi}\sin\varphi) 
- p_3\frac{1}{\lambda_{\CCCom{r}}\sin^2\varphi}(\dpart{\mu^{\rComM{c}}_L}{\varphi}\sin\varphi-\dpart{\mu^{\rComM{c}}_D}{\varphi}\cos\varphi),
\\
\dpart{\mathcal{L}_\l}{c_\l}	
	&= p_2\frac{1}{\sin^2\varphi}(\dpart{\mu^{\rComM{c}}_L}{{c_\l}}\cos\varphi+\dpart{\mu^{\rComM{c}}_D}{c_\l}\sin\varphi) 
	+ p_3\frac{1}{\lambda_{\CCCom{r}}\sin^2\varphi}(\dpart{\mu^{\rComM{c}}_L}{c_\l}\sin\varphi-\dpart{\mu^{\rComM{c}}_D}{c_\l}\cos\varphi).
\end{align*}
The associated optimization procedure is then formalized with Algorithm~\ref{opti_corrected}.
\begin{algorithm}
\caption{Numerical optimization}
\label{opti_corrected}
\begin{algorithmic}
\STATE{\textbf{Input:} $\textrm{Tol}>0$, $\kappa>0$, $\alpha\mapsto C_L(\alpha)$, $\alpha\mapsto C_D(\alpha)$, $\lambda_{\CCCom{r}}$,  $x\mapsto \psi(x)$.}\\
\STATE{\textbf{Initial guess:}  $\gamma_\l, c_\l$.}\\
\STATE{\textbf{Output:}  $\gamma_\l, c_\l$.}\\

\STATE Set $err \vcentcolon= \textrm{Tol}+1.$\;
\WHILE{$err>\textrm{Tol}$}{
\item Set $\varphi$, $a$, $a'$ as the solutions of~{\rm (\ref{Correc:Eq1}--\ref{Correc:Eq3})}.\;
\item Set $p$ as the solution of~\eqref{eq:p}.\;
\item Compute the gradient $\nabla J_\l(\gamma_\l, c_\l)$ given by~\eqref{eq:grad}.\;
\item Update $\begin{pmatrix}  \rComM{\gamma_\l} \\ \rComM{c_\l} \end{pmatrix} =\begin{pmatrix} \rComM{\gamma_\l} \\ \rComM{c_\l} \end{pmatrix} +\kappa\nabla J_\l(\gamma_\l, c_\l)$,\;
\item Set $err\vcentcolon=\|\nabla J_\l(\gamma_\l, c_\l)\|$.\;
}
\ENDWHILE
\end{algorithmic}
\end{algorithm}

\section{Numerical experiments}\label{numres}
In this section, we test the performance of the algorithms presented in Section~\ref{Sec:algo} on a practical case\ComM{s} and \ComM{tackle the design optimization problem considered in Section~\ref{Sec:opti} in the case of an actual wind turbine.}

\subsection{\ComM{Example of a small river turbine}}\label{numintro}
\ComM{This example is related to the project HyFloEFlu
, which was devoted to the design of a river turbine adapted to the Garonne river in Bordeaux, France. }
We consider a turbine of radius $R=1.1 $m, consisting of three blades, designed \ComM{with  a unique profile, namely NACA 4415. In the case we study, the TSR is close to 3, which corresponds for example to $U_{-\infty} =$1.5 m.s$^{-1}$ and $\omega=2/3.2\pi$s$^{-1}$ }. The 
 functions $C_L$ and $C_D$ have been obtained using  truncated Fourier representations of data provided by the free software Xfoil~\cite{Drela}, \rCCCom{ with $Re=9.10^5$.}
\rCCCom{
\begin{remark} Note that for such a value, Xfoil sometimes fails to predict airfoil lift and drag accurately because a very simple model is then used to describe the transition from laminar to turbulent flow in the airfoil boundary layer~\cite{vantreuren}. Hence, this test must be considered as an experiment to test the solution algorithms rather than an accurate estimate of the turbine efficiency.
\end{remark}}
 The first step of the usual design procedure presented in Section~\ref{sec:design} gives 
 $\overline{\alpha}=0.2215$ rad. Plots of $C_L$ and $C_D$ are given in Figure~\ref{polar}.

\begin{figure}[h!]
\begin{center}
	\begin{subfigure}{0.475\textwidth}
		\includegraphics[width=1\textwidth]{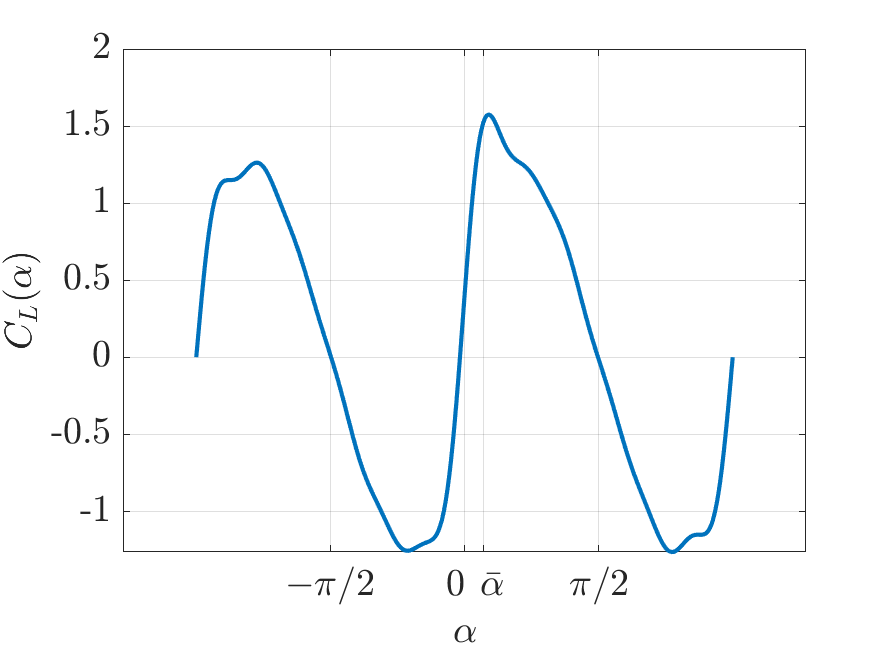}
	\end{subfigure}	
	\hfill
	\begin{subfigure}{0.475\textwidth}
		\includegraphics[width=1\textwidth]{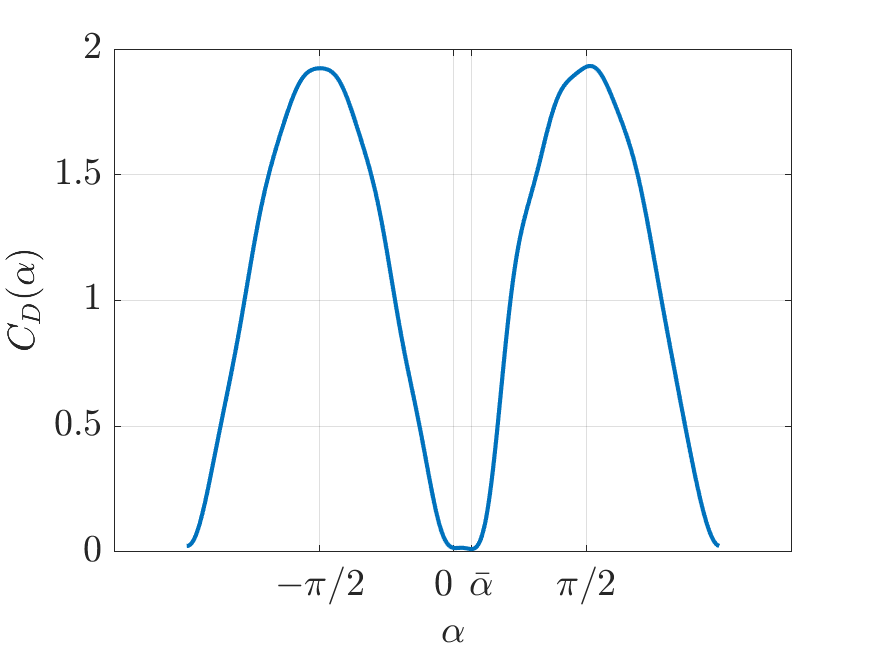}
	\end{subfigure}
\end{center}
\caption{Lift and Drag coefficients $C_L$ and $C_D$ as functions of $\alpha$, \rCCCom{for $Re=9.10^5$}.}\label{polar}
\end{figure}
We use the correction of Wilson \textit{et al} and Spera, meaning that $a_c=1/3$, see Table~\ref{Table_correc}. We first focus on three different blade elements associated with $\lambda_{\CCCom{r},1}=0.5$, $\lambda_{\CCCom{r},2}=1.75,$ and $\lambda_{\CCCom{r},3}=3$, respectively. 
In these three cases,  we  either set $(\gamma_\l,c_\l)=(\gamma_{\l}^{\ast},c_\l^{\ast})$, i.e. the optimal values of the simplified model given by~\eqref{opt_simp}
or $(\gamma_\l,c_\l)=:(\gamma_{\l}^{c},c_\l^{c})$, i.e. the optimal values of the corrected model. \ComM{The former typically corresponds to the first step of an optimization procedure, where $(\gamma_{\l}^{\ast},c_\l^{\ast})$ is used as an initial guess.
The  latter is computed with Algorithm~\ref{opti_corrected} and corresponds typically to one of the last steps of an optimization process}. 
 These values, as well as the associated $\varphi_c$ are given in Table~\ref{c_gamma_vals}, whereas corresponding graphs of the functions $\mu^c_{LD}:\varphi\mapsto\mu_L^c(\varphi)-\tan(\theta_\l-\varphi)\mu_D^c(\varphi)$, $\mu_{L}$, $\mu_G^c$ and $\mu_G$ are presented in Figure~\ref{fig:examples}. 
\begin{table}
\begin{center}
\begin{tabular}{|c||c|c||c|c||c|c|}
\hline
$\lambda_{\CCCom{r}}$ & $\gamma_\l^\ast$ & $\gamma_\l^c$ & $c_\l^\ast$ & $c_\l^c$ & $\varphi_c\quad (\gamma_\l^\ast,c_\l^\ast)$ & $\varphi_c\quad (\gamma_{\l}^{c},c_\l^{c})$ \\
\hline\hline
$\lambda_{\CCCom{r},1}=0.5 $  	& 0.516627 	& 0.520195	& 1.429701 &  0.255311	&    (0.715856,0.996060)    & 0.707018 \\ \hline 
$\lambda_{\CCCom{r},2}=1.75 $ 	& 0.124625	& 0.123680	& 0.318397 &  0.200768 	& 0.343358 & 0.343095\\ \hline 
$\lambda_{\CCCom{r},3}=3 $ 		& -0.006972	&  -0.052160	& 0.045246 &  0.066050  & 0.213839  & 0.214025\\ \hline 
\end{tabular}
\caption{Optimal values of the twist $\gamma_\l$ and the chord $c_\l$ for various values of $\lambda_{\CCCom{r}}$, with $(\gamma_\l,c_\l)=(\gamma_{\l}^{\ast},c_\l^{\ast})$ and $(\gamma_\l,c_\l)=:(\gamma_{\l}^{c},c_\l^{c})$. }\label{c_gamma_vals}
\end{center}
\end{table}
For these two blade geometries,  $I^+=(0,\theta_{\monlambda}]$, $\rComM{\varphi_c}<\theta_\l$ for all elements. 
However, the function $g$ is non-decreasing in few cases, as, e.g. when $\l=0.5$ and $(\gamma_\l,c_\l)=(\gamma_\l^\ast,c_\l^\ast)$. In this case, the last statement of Lemma~\ref{lem:implicit} 
does not apply which explains the existence of two values $\varphi_c^1$ and $\varphi_c^2$ where the graphs of $\mu_G^c$ and $\mu_G$ merge or separate. 
In the other cases, Corollary~\ref{cor:exist_corrected} applies, which is confirmed by the plots. 
\begin{figure}[h!]
\begin{center}
		\includegraphics[width=.32\textwidth]{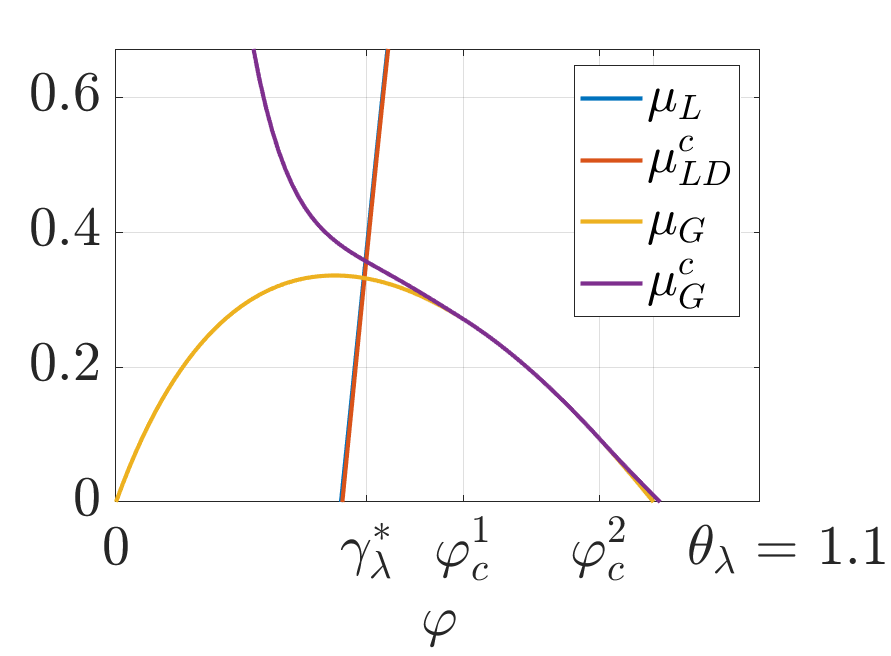}
		\includegraphics[width=.32\textwidth]{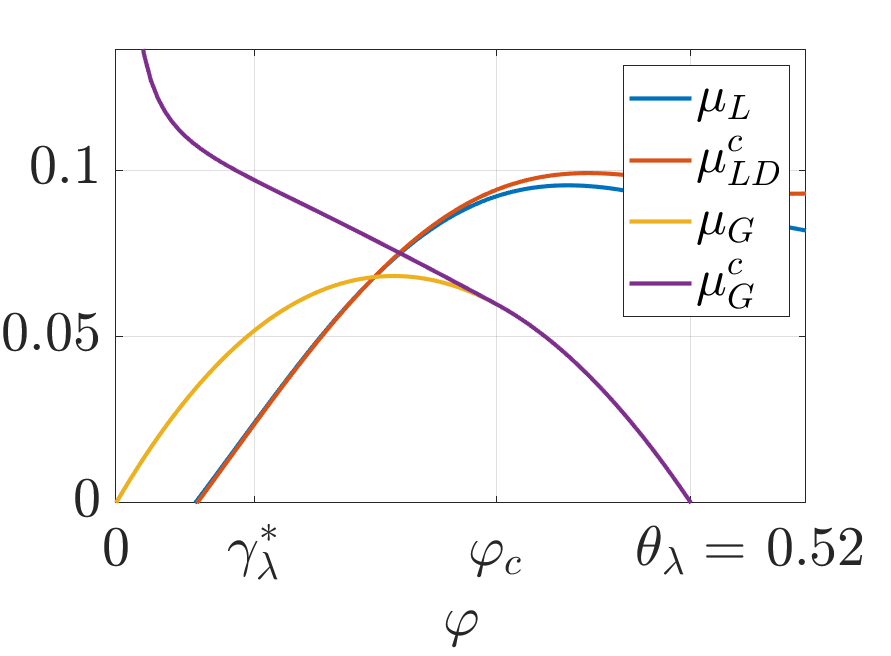}
		\includegraphics[width=.32\textwidth]{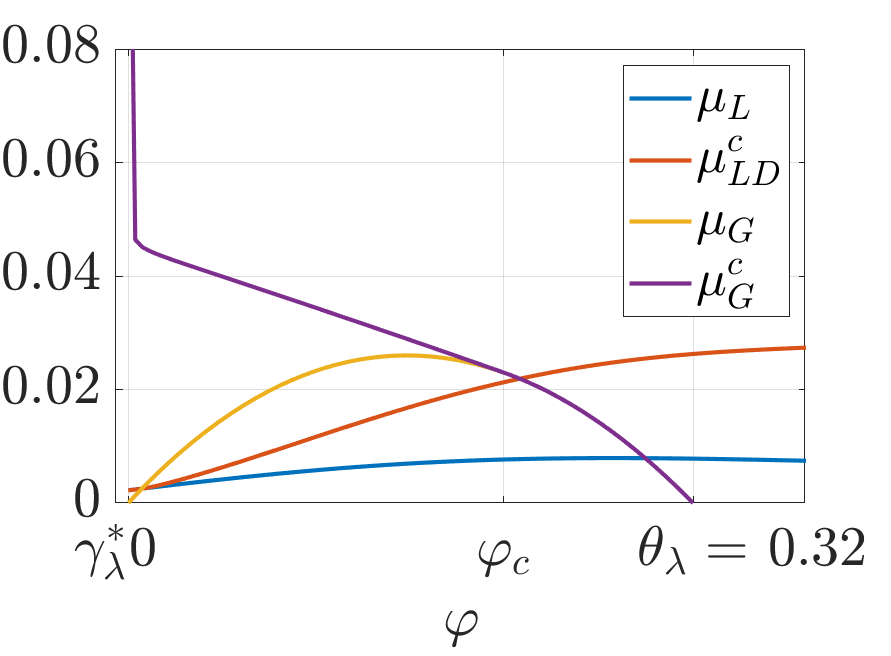}\\

		\includegraphics[width=.32\textwidth]{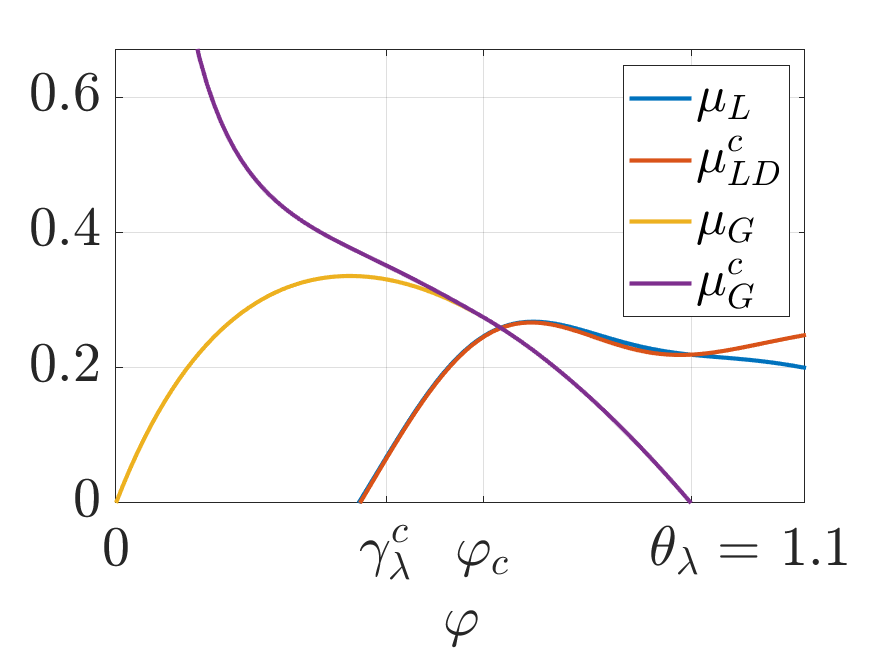}
		\includegraphics[width=.32\textwidth]{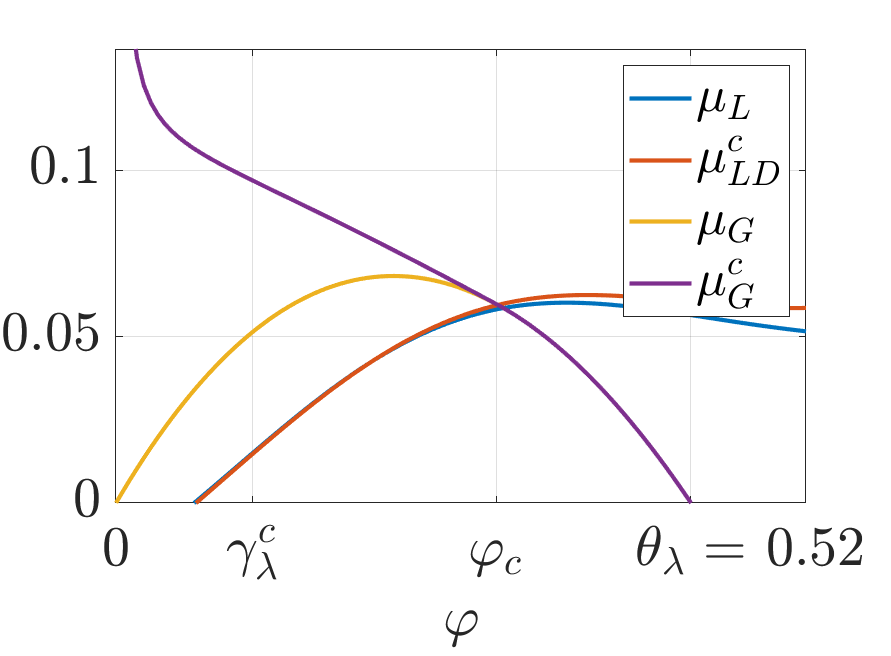}
		\includegraphics[width=.32\textwidth]{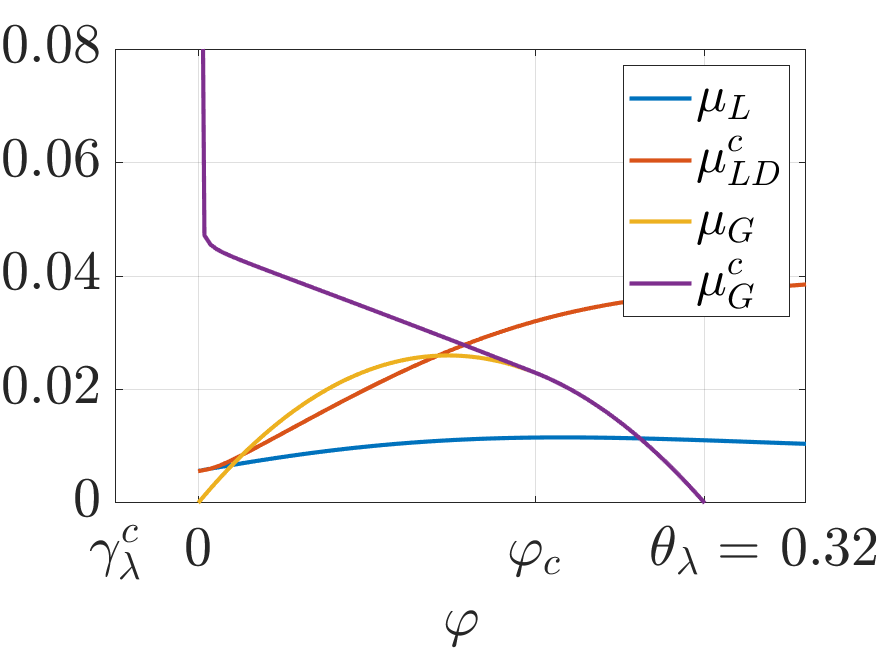}
	\caption
{Graphs of the functions $\mu_{L}$,  $\mu^c_{LD}$, $\mu_G^c$ and $\mu_G$, with $(\gamma_\l,c_\l)=(\gamma_\l^\ast,c_\l^\ast)$ (top) and $(\gamma_\l,c_\l)=(\gamma_\l^c,c_\l^c)$ (bottom), for $\lambda_{\CCCom{r}}=0.5$ (left), $\lambda_{\CCCom{r}}=1.75$ (middle), $\lambda_{\CCCom{r}}=3$ (right). The values of these parameters are given in Table~\ref{c_gamma_vals}. Note that these figures are similar to the scheme given in~\cite[Figure 3.27, p.126]{Manwell}.}
	\label{fig:examples}
\end{center}
\vspace{-.5cm}
\end{figure}
We observe multiple solution of type (1) (see Section~\ref{multiple}) 
 in the case of the simplified model. As for the full corrected model, we always have a unique solution. 

To compare the efficiency of the solution algorithms  we consider the corrected model and 
 measure the number of iterations ($k$) required
to solve accurately~\eqref{eq:corr_mu} in the sense that 
$\left| Res(\varphi^k) \right|
\leq \textrm{Tol} = 10^{-10}.$ 
We use this stopping criterion in all our tests, instead of the respective definitions of $err$ given in the algorithms.
\ComM{We test the algorithms presented in Section~\ref{Sec:algo}, i.e. the two versions of the standard fixed point, the robust and Newton versions of the optimized fixed point and the root-finding algorithms.}
Remark that due to the choice of 
 correction, iterations in \ComM{each} 
 algorithm have similar computational costs, namely, the solving of second order polynomials corresponding to~\eqref{Correc:Eq2} when applying  Algorithm~\ref{algo_corrected} or \ComM{Ning's algorithm 
and to~\eqref{eq:a=f} when applying Algorithm~\ref{algo_new_algo} or our new root-finding to solve $Res(\varphi)=0$ (see Section~\ref{Sec:RF}). 
}
In this test, \ComM{root-finding algorithms do} not always apply when $a_c=1$: such a case gives rise to multiple solution (see 
Figure~\ref{fig:examples}
) implying that $Res(\varphi)$ have the same sign on both sides of $I_0$. 
The initialization is done with $\varphi^0=\theta_\l$ for Algorithm~\ref{algo_corrected} and Algorithm~\ref{algo_new_algo}, whereas 
\ComM{the root-finding algorithms are initialized with the intervals $I_0:=[10^{-2},\pi/2-10^{-2}]$ for Ning's algorithm and $I_0:=[10^{-2},\theta_\l]$ for our new root-finding algorithm to solve $Res(\varphi)=0$}. 
We set $\gamma_\l=\gamma_\l^\ast$ and $c_\l=c_\l^\ast$ and run our test on the two cases $a_c=1$ and $a_c=1/3$.  The former case gives rise to a situation where $\psi((a-a_c)_+)=\psi(0)=0$, i.e., $\mu_G^c=\mu_G$, so that Theorem~\ref{th:cv} applies. 
The results are presented in Figure~\ref{tab:it}.
\begin{figure}[h!]
\begin{center}
	\includegraphics[width=.46\textwidth]{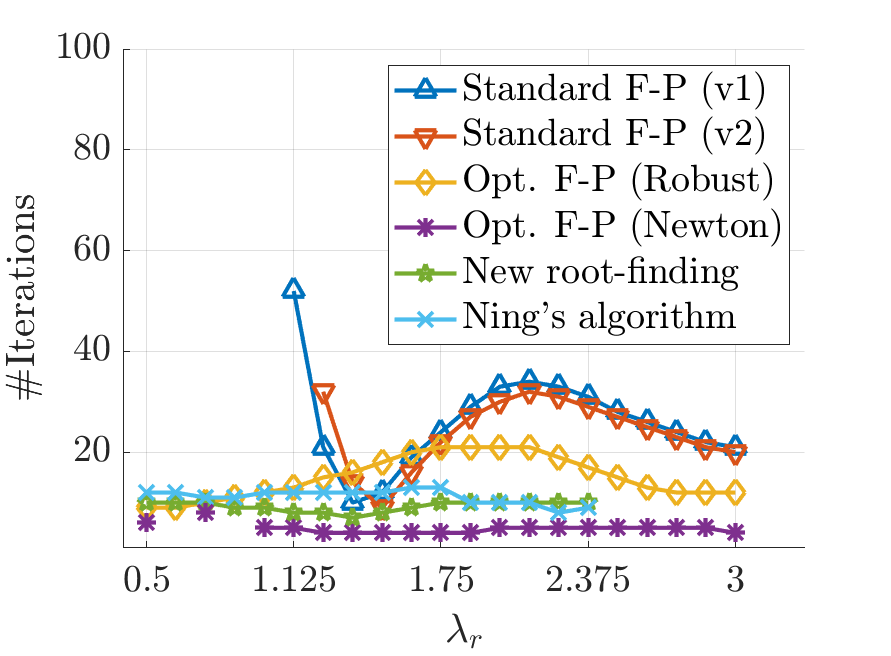}
	\includegraphics[width=.46\textwidth]{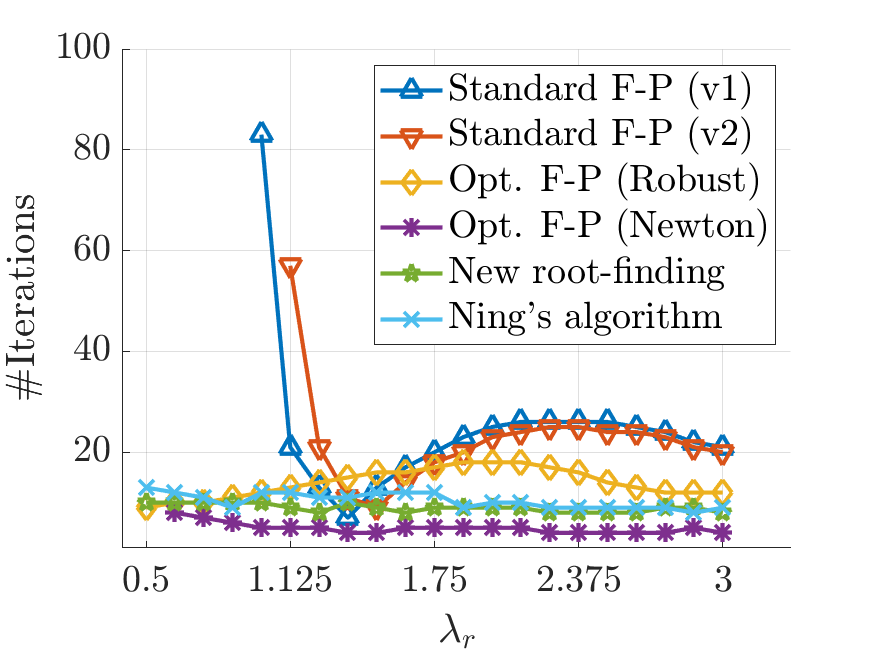}
\caption{Number of iterations required to solve~\eqref{eq:corr_mu} with  Algorithm~\ref{algo_corrected} (Standard fixed-point version 1 \& 2), Algorithm~\ref{algo_new_algo} (Optimized fixed-point, robust and Newton versions) 
  and root-finding approaches (Ning's algorithm and new root-finding)
as function of $\l$. Left: $a_c=1$, right: $a_c=1/3$. Only convergence cases are reported.
 }
\label{tab:it}
\end{center}
\vspace{-.5cm}
\end{figure}
\Com{In case of convergence, the obtained limit is the same with all algorithms}.
\ComM{We observe that the robust version of Algorithm~\ref{algo_new_algo} is the only algorithm that always converges and that its Newton version outperforms all other algorithms and only diverges in three cases.}

\subsection{\CCCom{Example of a large \rCCCom{wind turbine}}}\label{numNL}
\CCCom{We consider the IEA Wind 15-MW reference wind turbine\footnote{All numerical data related to this example are given in \url{https://github.com/IEAWindTask37}.}.
In this example the blade length is 117 m, decomposed into $50$ elements and designed for a TSR equal to 9. 
 The coefficients $C_L$ and $C_D$ take into account 3D-effects by using Du-Selig~\cite{Du1998} stall delay 3D correction and are specified for each element.  We interpolate them using the Akima algorithm~\cite{Akima70,Akima74} which provides smooth approximations. A full description of the turbine is given in~\cite{Evan}.


In our test, we compute optimal parameters $(\gamma_\l^c,c_\l^c)$ using the Matlab function \texttt{fminunc}, providing the gradient as in Algorithm~\ref{opti_corrected}. 
For the sake of consistency, with use the same tolerance Tol$=10^{-7}$ in the stopping criteria of the optimization procedure and of the solution algorithm considered to solve~\eqref{eq:mu_mu}, i.e., the algorithms stop when  $\|\nabla J_\l(\gamma_\l, c_\l)\|\leq$ Tol and $|Res(\varphi)|\leq$ Tol, respectively. For the sake of numerical efficiency, we use a continuation approach, meaning that we optimize the elements sequentially by starting from the blade tip and initialize the optimization of the current element with the design obtained for the element previously considered. This process is itself initialized with optimum approximation applied to the tip element. 
We then compare it to the actual design, denoted by $(\gamma_\l^{true},c_\l^{true})$ to $(\gamma_\l^c,c_\l^c)$. 
The results are presented in Figure~\ref{opti_numNL}, whereas the corresponding values of the power coefficient $C_P$ are 
$C_P(\gamma_\l^{true},c_\l^{true})=0.4531$\footnote{The actual design is planned to achieve $C_P=0.489$, see~\cite[p.8]{Evan}},  $C_P(\gamma_\l^c,c_\l^c)=0.4628$ and $C_P(\gamma_\l^\ast,c_\l^\ast)=0.0127$. 
\begin{figure}
\begin{center}
		\includegraphics[width=.32\textwidth]{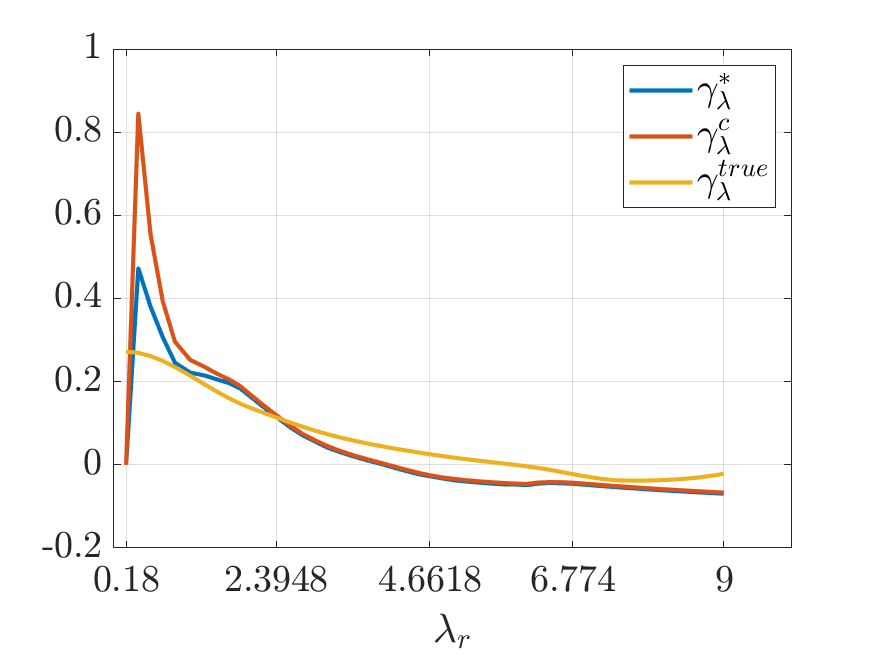}
		\includegraphics[width=.32\textwidth]{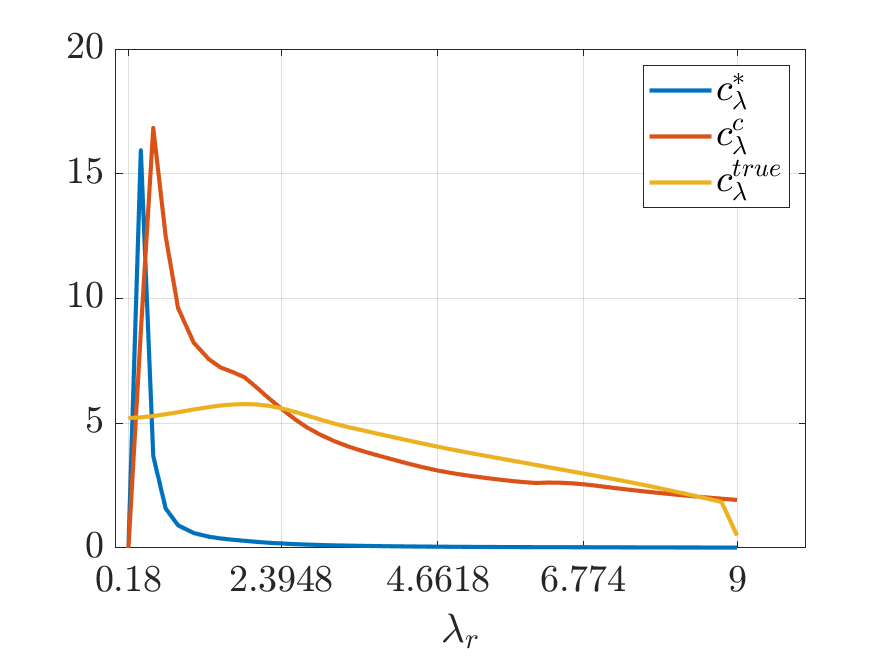}
		\includegraphics[width=.32\textwidth]{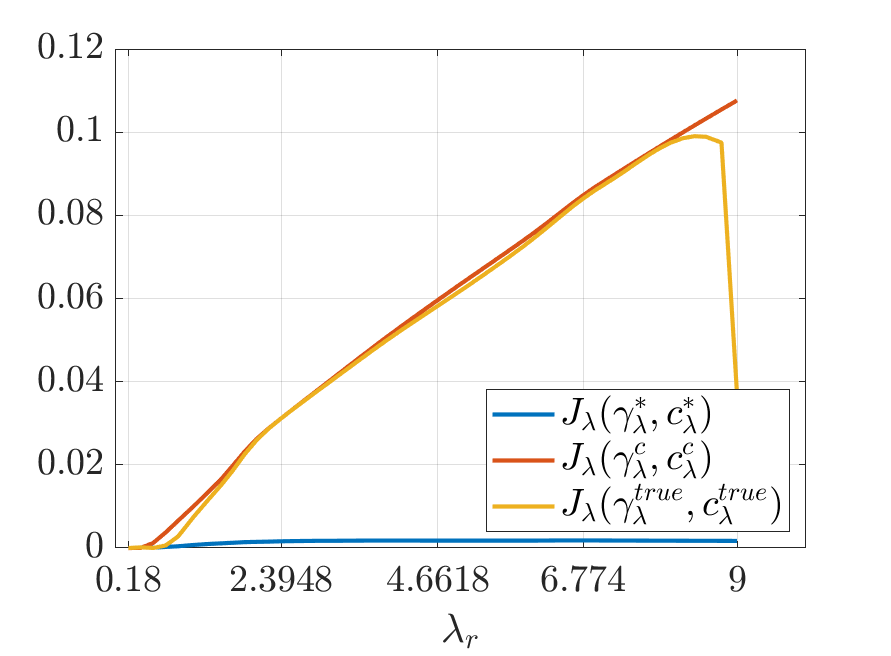}
	\caption{Graphs of $\gamma_\l^\ast$ and $\gamma_\l^c$, $c_\l^\ast$ and $c_\l^c$, 
 $J_\l(\gamma_\l^\ast,c_\l^\ast)$ and $J_\l(\gamma_\l^c,c_\l^c)$.}
	\label{opti_numNL}
\end{center}
\vspace{-.5cm}
\end{figure}
We observe that the optimum approximation gives a very bad design. 
On the other hand, the optimal design $(\gamma_\l^c,c_\l^c)$ appears to be close to the actual design $(\gamma_\l^{true},c_\l^{true})$, except for small values of $\l$, i.e. in the neighborhood of the hub where it only weakly influences the power coefficient. This explains that these two  designs give similar values of this coefficient.
 Remark that the chords $c_\l^\ast$ and $c_\l^c$ are very large for small values of $\l$, which makes them unrealistic in practice.
Though the results do not depend on the considered solution algorithm (up to the tolerance Tol), this one impacts the total number of iterations used to solve~\eqref{eq:mu_mu} required during the optimization process, as shown in Table~\ref{iter_opt}. 
The robust version and Newton version of Algorithm~\ref{algo_new_algo} appear to be the fastest procedures in this example.

\begin{table}
\begin{center}
\begin{tabular}{|c||c|c|c|}

\hline Solution algo.   & Stand. F-P (v1) & Stand. F-P (v2) & Opt. F-P (robust)  \\
\hline
\#Iterations &  9105 & 8646 & 5045   \\
\hline 
\end{tabular}

\medskip

\begin{tabular}{|c||c|c|c|}

\hline Solution algo.   &  Opt. F-P (Newton)  & New root-finding & Ning's alg.  \\
\hline
\#Iterations &   1688 & 5653 & 5965  \\
\hline 
\end{tabular}
\caption{Number of iterations of the solutions algorithms during the optimization. }\label{iter_opt}
\end{center}
\end{table}



}
\Com
{\section{Conclusion}
In this paper, we present a new formulation of the BEM model, which respects the paradigm of this approach in the sense that it decomposes the model into a macroscopic part, related to the momentum theory, and a local part, related to the blade element theory.
This framework allows us to obtain existence results and new solution algorithms which outperform the usual algorithms and whose convergence can be analyzed mathematically.
We have focused on the case of extracting turbines and a future work could consist in extending our formulation to propellers, which was the initial purpose of this theory. Moreover, some work is required to combine or include this model into modern CFD codes devoted to turbine design. Using BEM model as coarse solver or preconditionner of the fluid-structure interaction system could improve the convergence properties of associated PDE solvers.} 	

\section*{Acknowledgments}
The authors acknowledge support from ANR Cin\'e-Para (ANR-15-CE23-0019) and ANR HyFloEFlu (ANR-10-IEED-0006-04). \rComM{ The authors thank the anonymous referees for their insightful suggestions. J.S. thanks Dylan Machado for his careful proofreading of the article. 
}

\section*{Appendix: a case of convergence of Algorithm~\ref{algo_corrected} \ComM{(version 1)}}
\label{case_simple}

 In the case of the simplified model, Algorithm~\ref{algo_corrected} reads as an iterative procedure 
based on the formula
\begin{equation}\label{eq:iteration_standard}
\varphi^{k+1} :=\widetilde{f}(\varphi^k),
\end{equation}
where $\widetilde{f}(x):=\frac\pi 2 - {\rm atan}\left(\l+\mu_L(x)h(x)  \right)$ and $h(x):= \dfrac{\lambda_{\CCCom{r}}\tan^{-1}x +1}{        \sin x}$. 
This framework makes it possible to obtain bounds for this sequence.
\begin{lemma}\label{lem:stab}
Suppose that Assumption~\ref{assume:gal} holds and that $\max I^+=\theta_\l$, 
 with $\mu_L{}$ non-decreasing.
If 
\begin{equation}\label{assum:stability_standard}
\mu_L(\theta_\l)\leq \mu_G(\gamma_\l).
\end{equation} 
and $\varphi^0 \in[\gamma_\l,\theta_\l]$, then the sequence defined by~\eqref{eq:iteration_standard} satisfies $\forall k\in\N, \ \varphi^{k}\in [\gamma_\l,\theta_\l]$.
\end{lemma}

\begin{proof}
Assume that for some $k\in\N$, $\varphi^{k}\in  [\gamma_\l,\theta_\l]$.
Because of~\eqref{eq:iteration_standard}, we have $\tan^{-1}\varphi^{k+1} :=\l+\mu_L(\varphi^k) h(\varphi^k)$.
Combining it with 
 $\mu_L\geq 0$  
 gives $\varphi^{k+1}\leq \theta_\l$. 
Since $\mu_L$ is increasing and $h$ is decreasing, $\tan^{-1}\varphi^{k+1} \leq \l+\mu_L(\theta_\l)h(\gamma_\l)$.
Because of~\eqref{assum:stability_standard}, the latter is bounded by $\tan^{-1}\gamma_\l$. The result follows by induction. 
\end{proof}

We complete this result by a condition about a contraction property.
\begin{lemma}\label{lem:der}
Suppose that $\mu_L$ is differentiable and denote by $\mu_L'$ its derivative.  The derivative of $\widetilde{f}$ satisfies
\[ -\sin\theta_\l\max\limits_{I^+}\mu_L' h(\gamma_\l) \leq \widetilde{f}{}'(\varphi) \leq \sin\theta_\l\max\limits_{I^+}\mu_L |h'(\gamma_\l)|.\]
\end{lemma}

\begin{proof}
We have $\widetilde{f}{}'(\varphi)= \frac{-1}{1+\left(\l+\mu_L(\varphi)h(\varphi)\right)^2}(\mu_L'(\varphi)h(\varphi)+\mu_L(\varphi)h'(\varphi))$.
Since $\varphi\in (\gamma_\l,\theta_\l)$, $\mu_L'(\varphi)h(\varphi)\geq 0$, $\mu_L(\varphi)h'(\varphi)\leq 0$, we have $\frac{-1}{1+\left(\l+\mu_L(\varphi)h(\varphi)\right)^2}\mu_L'(\varphi)h(\varphi) 
\leq \widetilde{f}{}'(\varphi)\leq
\frac{-1}{1+\left(\l+\mu_L(\varphi)h(\varphi)\right)^2} \mu_L(\varphi)h'(\varphi)$. 
Because $h$ and $h'$ are decreasing on $(\gamma_\l,\theta_\l)$, and since $\mu_L(\varphi)h(\varphi)\geq 0$, the result follows.
\end{proof}

We are now in a position to obtain a conditional convergence result.

\begin{theorem}\label{th:classic}
In addition to the assumptions of Lemma~\ref{lem:stab}, suppose that $\mu_L$ is differentiable and satisfies
\begin{equation}
\sin\theta_\l\max\limits_{I^+} \mu_L'  h (\gamma_\l)  \leq 1,\qquad  
\sin\theta_\l\max\limits_{I^+} \mu_L   |h'(\gamma_\l)| \leq 1.\label{contract2}
\end{equation}
Then, if $\varphi^0$ belongs to $[\gamma_\l,\theta_\l]$, the sequence $(\varphi^k)_{k\in\N}$ defined by~\eqref{eq:iteration_standard} converges to the unique solution of~\eqref{eq:mu_mu}.
\end{theorem}
\begin{proof}
As a consequence of Lemma~\ref{lem:stab}, the function $\widetilde{f}$ maps $[\gamma_\l,\theta_\l]$ onto itself. From
~\eqref{contract2} and Lemma~\ref{lem:der} we deduce $\widetilde{f}$ is contracting. The result follows from the Banach fixed-point theorem.
\end{proof}

\bibliographystyle{abbrv}			
\bibliography{biblio}

\end{document}